\documentclass[12pt]{amsart}

\usepackage{amsmath,amssymb,amscd,amsthm,graphicx,enumerate}
\usepackage[usenames,dvipsnames]{xcolor}

\usepackage{hyperref}
\hypersetup{breaklinks=true}

\usepackage{enumitem} 
\newlist{condenum}{enumerate}{1} 
\setlist[condenum]{label=\bfseries Condition \arabic*.,  ref=\arabic*, wide}

\numberwithin{equation}{section}
\theoremstyle{plain}



\usepackage[utf8]{inputenc}
\usepackage{amssymb}
\usepackage{amsthm}
\usepackage{mathrsfs}
\usepackage{amsfonts}
\usepackage{amsmath}
\usepackage{mathtools}
\usepackage{makecell}
\usepackage{xcolor}
\usepackage[margin=1.25in]{geometry}

\usepackage{lipsum}
\makeatletter
\def\ps@pprintTitle{%
 \let\@oddhead\@empty
 \let\@evenhead\@empty
 \def\@oddfoot{}%
 \let\@evenfoot\@oddfoot}
\makeatother

\newcommand{\Rm}{\textnormal{Rm}}

\newcommand{\Ric}{\textnormal{Ric}}

\newcommand{\rii}{\rightarrow\infty}
\newcommand{\ri}{\rightarrow}
\newcommand{\AVR}{\textnormal{AVR}}
\newcommand{\diam}{\textnormal{diam}}
\newcommand{\sy}{\mathbb{Z}_2\times O(2)}
\newcommand{\cigar}{\textnormal{Cigar}}
\newcommand{\gs}{\Gamma(s)}
\numberwithin{equation}{section}

\newtheorem{theorem}{Theorem}[section]
\newtheorem{lem}[theorem]{Lemma}

\newtheorem{cor}[theorem]{Corollary} 

\theoremstyle{definition}
\newtheorem{defn}[theorem]{Definition}
\newtheorem*{theorem*}{Theorem}

\usepackage{xpatch}
\makeatletter
\xpatchcmd{\tableofcontents}{\contentsname \@mkboth}{\small\contentsname \@mkboth}{}{}
\xpatchcmd{\listoffigures}{\chapter *{\listfigurename }}{\chapter *{\small\listfigurename }}{}{}
\makeatother

\begin{document}

\begin{abstract}
We find a family of 3d steady gradient Ricci solitons that are flying wings. This verifies a conjecture by Hamilton. For a 3d flying wing, we show that the scalar curvature does not vanish at infinity. The 3d flying wings are collapsed.

For dimension $n\ge 4$, we find a family of $\mathbb{Z}_2\times O(n-1)$-symmetric but non-rotationally symmetric n-dimensional steady gradient solitons with positive curvature operator. We show that these solitons are non-collapsed.
\end{abstract}

\title[A family of 3d steady gradient solitons]{A family of 3d steady gradient solitons that are flying wings}

\author[Yi Lai]{Yi Lai}
\email{yilai@berkeley.math.edu}
\address[]{Department of Mathematics, University of California, Berkeley, CA 94720, USA}

\maketitle


\setcounter{tocdepth}{1}
%

\begin{section}{Introduction}\label{s: intro}
Ricci solitons are self-similar solutions of the Ricci flow equation, and they often arise as singularity models of Ricci flows.
In particular, a steady gradient soliton is a smooth complete Riemannian manifold $(M,g)$ satisfying
\begin{equation}
    \Ric=\nabla^2 f
\end{equation}
for some smooth function $f$ on $M$, which is called a potential function.
The soliton generates a Ricci flow for all time by $g(t)=\phi_t^*(g)$, where $\{\phi_t\}_{t\in(-\infty,\infty)}$ is the one-parameter group of diffeomorphisms generated by $-\nabla f$ with $\phi_0$ the identity.

In dimension 2, the only non-flat rotationally symmetric steady gradient soliton is Hamilton's cigar soliton \cite{cigar}. In any dimension $n\ge3$, the only non-flat rotationally symmetric steady gradient soliton is the Bryant soliton, which is constructed by Bryant \cite{bryant}.
It is an open problem whether there are any 3d steady gradient solitons other than the 3d Bryant soliton and quotients of $\mathbb{R}\times\cigar$, see e.g. \cite{infinitesimal,Catino,Chow2007a,DZ}.

Hamilton conjectured that there exists a 3d flying wing, which is a $\sy$-symmetric 3d steady gradient soliton asymptotic to a  sector with angle $\alpha\in(0,\pi)$. The term flying wing is also used by Hamilton to describe certain translating solutions in mean curvature flow.
A lot of important progress has been made
for the mean curvature flow flying wings in the past two decades.
For example, the flying wings in $\mathbb{R}^3$ are completely classified by the works of X.J. Wang \cite{Wangxujia} and Hoffman-Ilmanen-Martin-White \cite{white}. Moreover, higher dimensional examples were constructed independently by Bourni-Langford-Tinaglia \cite{langford1} and Hoffman-Ilmanen-Martin-White \cite{white}. 

Despite many analogies between the Ricci flow and mean curvature flow, Hamilton's flying wing conjecture remains open. 
A proposed approach is to obtain the flying wings as limits
of solutions of elliptic boundary value problems. 
This is how the flying wings in mean curvature flow are constructed, where the solutions can be parametrized as graphs \cite{Wangxujia}. However, it seems hard to choose such a parametrization in Ricci flow to get a strictly elliptic equation.  
In this paper, we confirm Hamilton's conjecture by using a different approach.

Our first theorem finds a family of non-rotationally symmetric $n$-dimensional steady gradient solitons with prescribed Ricci curvature at a point in all dimensions $n\ge3$.
This gives an affirmative answer to the open problem by Cao whether there exists a non-rotationally symmetric steady Ricci soliton in dimensions $n\ge 4$ \cite{CaoHD}.
Throughout this section, the quadruple $(M,g,f,p)$ denotes a steady gradient soliton, where $f$ is the potential function and $p$ is a critical point of $f$.
\begin{theorem}\label{t: existence with prescribed eigenvalue}
Given any $\alpha\in(0,1)$, there exists an n-dimensional $\mathbb{Z}_2\times O(n-1)$-symmetric steady gradient soliton $(M,g,f,p)$ with positive curvature operator, such that $\lambda_1=\alpha\lambda_2=\dots=\alpha\lambda_n$, where $\lambda_1,\dots,\lambda_n$ are eigenvalues of the Ricci curvature at p.  
\end{theorem}

The 3d steady gradient solitons from Theorem \ref{t: existence with prescribed eigenvalue} are collapsed, which is an easy consequence of its asymptotic geometry. This also follows from the uniqueness of the Bryant soliton among 3d non-collapsed steady gradient solitons by Brendle \cite{brendlesteady3d}.
Moreover, we show that the n-dimensional steady gradient solitons from Theorem \ref{t: existence with prescribed eigenvalue} are non-collapsed for all $n\ge4$.
They are analogous to the non-collapsed translators in mean curvature flow constructed by Hoffman-Ilmanen-Martin-White \cite{white}.


Our second theorem says that a $\mathbb{Z}_2\times O(2)$-symmetric 3d steady gradient soliton must be a Bryant soliton if the asymptotic cone is a ray. So the family of 3d steady gradient solitons from Theorem \ref{t: existence with prescribed eigenvalue} are all flying wings, which confirms Hamilton's conjecture. Figure 1 is the picture of a 3d flying wing. 

\begin{figure}[h]\label{f}
\centering
\includegraphics[width=12.5cm,height=5.5cm]{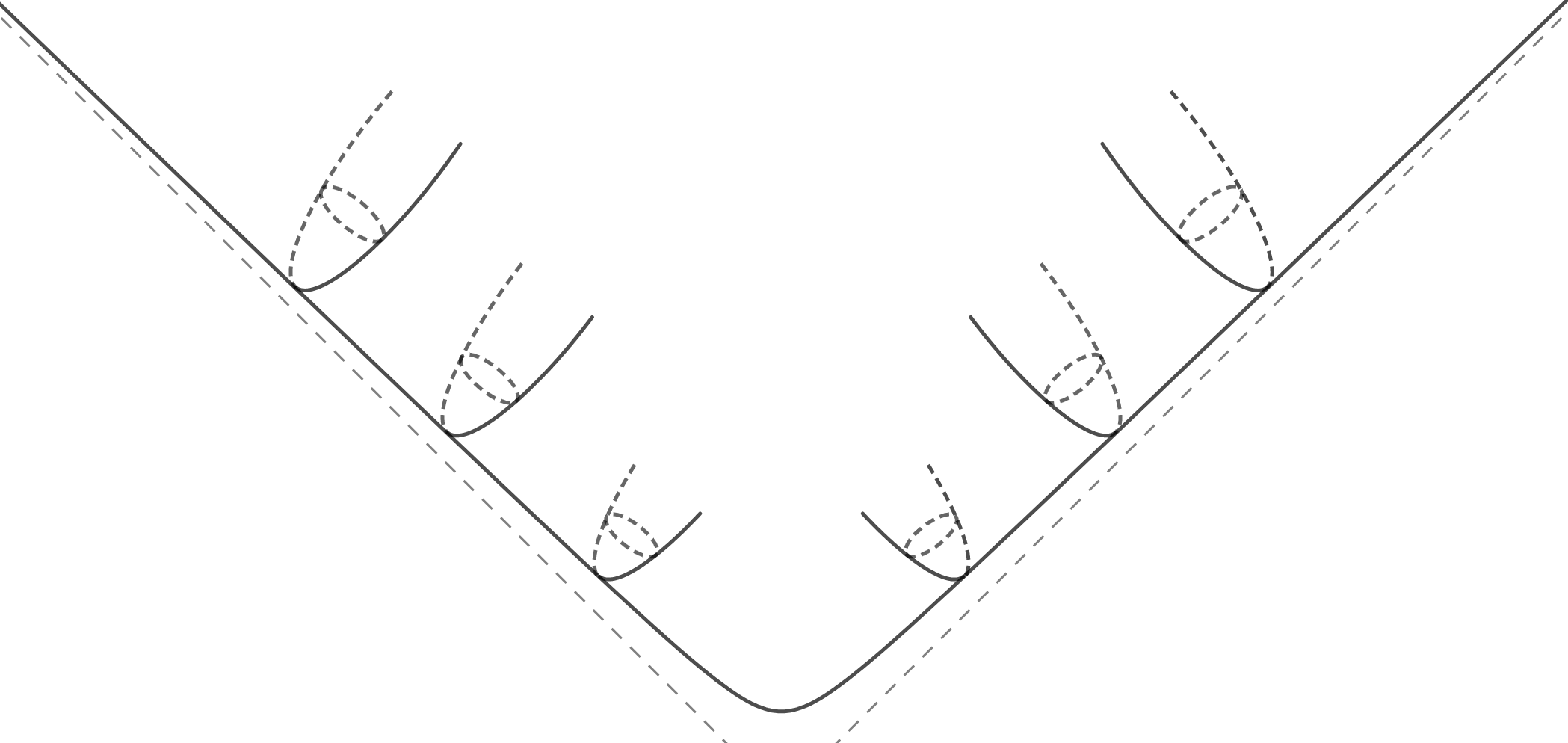}
\caption{A 3d flying wing} 
\end{figure}

\begin{theorem}\label{t: a ray implies Bryant soliton}
Let $(M,g,f,p)$ be a $\mathbb{Z}_2\times O(2)$-symmetric 3d steady gradient soliton. Suppose its asymptotic cone is a ray. Then it is isometric to the Bryant soliton.
\end{theorem}



\begin{cor}\label{c: unique of bryant}
A $\sy$-symmetric but non-rotationally symmetric 3d steady gradient soliton with positive curvature operator is a flying wing. In particular, the 3d steady gradient solitons from Theorem \ref{t: existence with prescribed eigenvalue} are all flying wings.
\end{cor}

It has been wondered whether the scalar curvature vanishes at infinity in all 3d steady gradient solitons. By Theorem \ref{l: positive angle} we see that this fails in 3d flying wings.
More precisely, Theorem \ref{l: positive angle} shows that the scalar curvature has a positive limit along the edges of the wing, and there is a quantitative relation between this limit and
the angle of the asymptotic cone.

\begin{theorem}\label{l: positive angle}
Let $(M,g,f,p)$ be a $\sy$-symmetric 3d steady gradient soliton, whose asymptotic cone is a metric cone over the interval $[-\frac{\alpha}{2},\frac{\alpha}{2}]$ for some $\alpha\in[0,\pi]$.
Let $\Gamma:(-\infty,\infty)\ri M$ be the complete geodesic fixed by the $O(2)$-action, then
\begin{equation}
    \lim_{s\rii}R(\Gamma(s))=R(p)\sin^2\frac{\alpha}{2}.
\end{equation}
\end{theorem}

We prove in the following corollary that the asymptotic geometry of a 
3d flying wing is uniquely determined by the angle of the asymptotic cone.
In particular, it converges to $\mathbb{R}\times\cigar$ along the edges. 
This is analogous to mean curvature flow flying wings, where the asymptotic geometry is uniquely determined by the width of the slab that contains the wing \cite{langford1}.

\begin{cor}\label{c: aymptotic geometry}
Let $(M,g,f,p)$ be a 3d flying wing, whose asymptotic cone is a sector with angle $\alpha\in(0,\pi)$. Then for any sequence of points $q_i\in\Gamma$ going to infinity, the sequence of pointed Riemannian manifolds $(M,g,q_i)$ smoothly converges to  $\mathbb{R}\times\cigar$, where the scalar curvature at the tip of the cigar is  $R(p)\sin^2\frac{\alpha}{2}$.
\end{cor}

As an application of Theorem \ref{t: a ray implies Bryant soliton} and \ref{l: positive angle}, we construct a sequence of 3d flying wings whose asymptotic cones have arbitrarily small angles.


\begin{cor}\label{c: alpha_i}
There exists a sequence of 3d flying wings $\{(M_i,g_i)\}_{i=1}^{\infty}$, whose asymptotic cone is a sector with angle $\alpha_i\in(0,\pi)$ such that $\lim_{i\rii}\alpha_i=0$.
\end{cor}

The structure of the paper is as follows.
In Section \ref{s: existence}, we prove Theorem \ref{t: existence with prescribed eigenvalue} by obtaining the steady gradient solitons as limits of appropriate expanding gradient solitons, whose construction is based on Deruelle's results \cite{De15}.
More specifically, we choose a sequence of expanding gradient solitons whose asymptotic volume ratio goes to zero, and prove that by passing to a subsequence they converge to a steady gradient soliton. 
In dimension 3, the sequence of expanding gradient solitons is between two sequences converging respectively to the 3d Bryant soliton and $\mathbb{R}\times\cigar$.

In Section \ref{s: construction}, we study the asymptotic geometry of $\sy$-symmetric 3d steady gradient solitons that are not Bryant solitons.
We prove a dimension reduction theorem which shows that the soliton smoothly converges to $\mathbb{R}\times\cigar$ at infinity.
We also show that the higher dimensional solitons from Theorem \ref{t: existence with prescribed eigenvalue} are non-collapsed.

In Section \ref{s: proof}, we first prove Theorem \ref{l: positive angle} and then use it to prove Theorem \ref{t: a ray implies Bryant soliton} and all the corollaries.
To prove
Theorem \ref{l: positive angle}, we study the variations of $\nabla f$ along certain minimizing geodesics. By the soliton equation this amounts to computing the integral of the Ricci curvature along the geodesics.
Then Theorem \ref{l: positive angle} follows by estimating this integral. 
Our main tools are the dimension reduction theorem, curvature comparison arguments, and Perelman's curvature estimates for Ricci flows with non-negative curvature operator.

Theorem \ref{t: a ray implies Bryant soliton} is proved by a bootstrap argument. Suppose the soliton is not a Bryant soliton. So the dimension reduction theorem applies.
By the $\sy$-symmetry, the soliton away from the edges is a warped-product metric with $S^1$-fibers.
First, by using the dimension reduction theorem and some computations we obtain an estimate on the length of the $S^1$-fibers, which shows that it increases slower than the square root of the distance to the critical point.  

Second, by using the estimate from the first step and similar computations we obtain a better estimate, which shows that the length function stays bounded at infinity. Since the length function is concave by the non-negativity of the curvature, this implies that the scalar curvature does not vanish along the edges. This by Theorem \ref{l: positive angle} contradicts the assumption that the asymptotic cone is a ray, hence proves Theorem \ref{t: a ray implies Bryant soliton}.

I thank my PhD advisor Richard Bamler for inspiring discussions and comments. 
I also thank John Lott, Bennet Chow, Robert Haslhofer, Alix Deruelle, Mat Langford and Guoqiang Wu for valuable comments.

\end{section}

\begin{section}{A family of non-rotaionally symmetric steady gradient solitons}\label{s: existence}
The main result in this section is Theorem \ref{t: existence with prescribed eigenvalue}.
The outline of the proof is as follows.
We first construct a sequence of smooth families of expanding gradient solitons $\{(M_{i,\mu},g_{i,\mu},p_{i,\mu}),\mu\in[0,1]\}_{i=0}^{\infty}$ with positive curvature operator, such that
$(M_{i,0},g_{i,0},p_{i,0})$ converges to a Bryant soliton, and $(M_{i,1},g_{i,1},p_{i,1})$ converges to the product of $\mathbb{R}$ and an (n-1)-dimensional Bryant soliton if $n\ge4$, or a cigar soliton if $n=3$. 
Moreover, we require that the asymptotic volume ratio of each expanding gradient solitons tends to zero uniformly as $i\rii$. 

Let $\alpha_i(\mu)$ be the quotients of the smallest and largest eigenvalues of the Ricci curvature at $p_{i,\mu}$ in $(M_{i,\mu},g_{i,\mu},p_{i,\mu})$, then $\alpha_i(\mu)$ is a smooth function in $\mu$ for each fixed $i$.
Then for any $\alpha\in(0,1)$, there is some $\mu_i\in(0,1)$ such that $\alpha_i(\mu_i)=\alpha$.
Since the asymptotic volume ratio of $(M_{i,\mu_i},g_{i,\mu_i},p_{i,\mu_i})$ goes to zero, we can show that it subconverges to an n-dimensional steady gradient soliton $(M,g,p)$ with positive curvature operator. In particular, the quotients of the smallest and largest eigenvalues of the Ricci curvature at $p$ in $(M,g,p)$ is equal to $\alpha$.

To construct the expanding gradient solitons we use
Deruelle's work \cite{De15}. He showed that for any $(n-1)$-dimensional smooth simply connected Riemannian manifold $(X_1,g_{X_1})$ with $\Rm>1$, there exists a unique expanding gradient soliton $(M_1,g_1,p_1)$ with positive curvature operator that is asymptotic to the cone $(C(X_1),dr^2+r^2 g_{X_1})$. Moreover,
there is a one-parameter smooth family of expanding gradient solitons connecting $(M_1,g_1,p_1)$ to an expanding gradient soliton $(M_0,g_0,p_0)$, whose asymptotic cone is rotationally symmetric. By Chodosh's work the soliton $(M_0,g_0,p_0)$ is rotationally symmetric, and hence is a Bryant expanding soliton \cite{chodosh}.

\subsection{Preliminaries}
In this subsection we fix some notions that will be frequently used. 
First, we recall some standard notions and facts from Alexandrov geometry:
Let $(M,g)$ be a non-negatively curved Riemannian manifold, then for any triple of points $o,p,q\in M$, the comparison angle $\widetilde{\measuredangle}poq$ is the corresponding angle formed by minimizing geodesics with lengths equal to $d(o,p),d(o,q),d(p,q)$ in Euclidean space.
Let $op,oq$ be two minimizing geodesics in $M$ between $o,p$ and $o,q$, and $\measuredangle poq$ be the angle between them at $o$, then $\measuredangle poq\ge\widetilde{\measuredangle}poq$.
Moreover, for any $p'\in op$ and $q'\in oq$, the monotonicity of angle comparison implies $\widetilde{\measuredangle}p'oq'\ge \widetilde{\measuredangle}poq$.  

For a non-negatively curved Riemannian manifold $(M,g,p)$ and two rays $\gamma_1,\gamma_2$ with unit speed starting from $p$, the limit $\lim_{r\rii}\widetilde{\measuredangle}\gamma_1(r)p\gamma_2(r)$ exists and we say it is the angle at infinity between $\gamma_1$ and $\gamma_2$.
Moreover, the space $(X,d_X)$ of equivalent classes of rays is a compact length space, where two rays are equivalent if and only if the angle at infinity between them is zero, and the distance between two rays is the limit of the angle at infinity between them.
The asymptotic cone is a metric cone over the space of equivalent classes of rays, and it is isometric to the Gromov-Hausdorff limit of any blow-down sequence of the manifold, see e.g. \cite{KL}.

Next, we define what we mean by a Riemannian manifold to be $\mathbb{Z}_2\times O(n-1)$-symmetric.
First, we define an $O(n-1)$-action on the Euclidean space $\mathbb{R}^n=\{(x_1,...,x_n):x_i\in\mathbb{R}\}$, by extending the standard $O(n-1)$-action on $\mathbb{R}^{n-1}=\{x_n=0\}\subset\mathbb{R}^n$ in the way such that it fixes the $x_n$-axis.
Then we define a $\mathbb{Z}_2\times O(n-1)$-action on $\mathbb{R}^n$ by futhermore defining a $\mathbb{Z}_2$-action to be generated by a reflection that fixes the hypersurface $\{x_n=0\}$. 

Let $\Gamma_0=\{x_1=\cdots=x_{n-1}=0\}$, $N_0=\{x_1=\cdots=x_{n-2}=0,x_{n-1}> 0\}$ and $\Sigma_0=\{x_n=0\}$.
Then $\Gamma_0$ is the fixed point set of the $O(n-1)$-action, $\Sigma_0$ is the fixed point set of the $\mathbb{Z}_2$-action, and $N_0$ is one of the two connected components of the fixed point set of a subgroup isomorphic to $O(n-2)$.
\begin{defn}\label{d: symmetry}
We say that an $n$-dimensional Riemannian manifold $(M^n,g)$ is  $\mathbb{Z}_2\times O(n-1)$-symmetric if there exist an isometric  $\mathbb{Z}_2\times O(n-1)$-action, and a diffeomorphism $\Phi: M^n\rightarrow \mathbb{R}^n$ such that $\Phi$ is equivariant with the two actions, where the action on $\mathbb{R}^n$ is defined as above.

Let $\Gamma=\Phi^{-1}(\Gamma_0)$, $\Sigma=\Phi^{-1}(\Sigma_0)$, and $N=\Phi^{-1}(N_0)$. Then it is easy to see that
\begin{enumerate}
    \item $\Gamma$ is a geodesic that goes to infinity at both ends.
    \item $\Sigma$ is a rotationally symmetric $(n-1)$-dimensional totally geodesic submanifold. 
    \item $N$ is a totally geodesic surface diffeomorphic to $\mathbb{R}^2$. 
    \item $\Phi^{-1}(0)$ is the unique fixed point of the $\mathbb{Z}_2\times O(n-1)$-action, at which $\Gamma$ intersects orthogonally with $\Sigma$.
\end{enumerate}
Moreover, consider the projection $\pi:M\rightarrow N$, which maps a point $x\in M$ to a point $y\in N$, which is the image of $x$ under some action in $O(n-1)$. Equip $N$ with the induced metric $g_N$, then $\pi$ is a Riemannian submersion, and $N$ is an integral surface of the horizontal distribution. So there is a smooth positive function $\varphi:N\rightarrow\mathbb{R}$ such that $g=g_N+\varphi^2 g_{S^{n-2}}$ on $M\setminus\Gamma$, where $g_{S^{n-2}}$ is the standard round metric on $S^{n-2}$.

\end{defn}

In this paper, 
we study n-dimensional expanding or steady gradient soliton $(M^n,g)$ with non-negative curvature operator, whose potential function $f$ has a critical point $p$. We denote it by a quadruple $(M^n,g,f,p)$ (and sometimes a triple $(M^n,g,p)$). 
Note that $R$ attains its maximum at $p$ by the identity $R+|\nabla f|^2=\textnormal{const.}$, and $p$ is the unique critical point of $f$ if $\Rm>0$.

We assume $(M^n,g,f,p)$ is $\mathbb{Z}_2\times O(n-1)$-symmetric, and fix the notions $\Gamma,N,\varphi,\Sigma$ from above, and assume $\Gamma:(-\infty,\infty)\rightarrow M$ has unit speed and $\Gamma(0)=p$.
Assume $\Rm>0$. Then it is easy to see that $p$ is the unique point fixed by the $\mathbb{Z}_2\times O(n-1)$-action. Moreover, by the soliton equation $\nabla^2 f=\Ric+cg$, $c\ge0$, it follows that the potential function $f$ is invariant under the actions. So the geodesic
$\Gamma$, and all the unit speed geodesics in $\Sigma$ starting from $p$ are integral curves of $\frac{\nabla f}{|\nabla f|}$.

Moreover, use $i,j,k,l$ for indices on $N$, and $\alpha,\beta$ and $g_{\alpha\beta}$ for indices and metric components on $S^{n-2}$ with the standard round metric with radius one. Then by a computation the nonzero components of the curvature tensor of $(M\setminus\Gamma,g)$ are
\begin{equation}\label{e: computation}
   \begin{split}
       R_{ijkl}^M =R_{ijkl}^N,\quad
       R^M_{i\alpha\beta j}=-g_{\alpha\beta}(\varphi \nabla^2_{i,j}\varphi),\quad R^M_{\alpha\beta\beta\alpha}=(1-|\nabla\varphi|^2)\varphi^2(g_{\alpha\alpha}g_{\beta\beta}-g_{\alpha\beta}^2).
   \end{split} 
\end{equation}
So by $\Rm\ge0$ and the second equation we have $\nabla^2\varphi\le 0$ and $\varphi$ is concave.

\subsection{Proof of Theorem \ref{t: existence with prescribed eigenvalue}}



To prove Theorem \ref{t: existence with prescribed eigenvalue}, we will take a limit of a sequence of expanding gradient solitons with $R(p)=1$, where $p$ is the critical point of the potential function. To do this, we need an injectivity radius lower bound and a uniform curvature bound.
The curvature bounds follows directly from $R_{\max}=R(p)=1$, and the injectivity radius estimate follows from
the next lemma.

\begin{lem}\label{l: v_0}
There exists $C>0$ such that the following holds: Let $(M^n,g,f,p)$ be a $\mathbb{Z}_2\times O(n-1)$-symmetric n-dimensional expanding (or steady) gradient soliton with positive curvature operator. Suppose $R(p)=1$. Then $vol(B(p,1))\ge C^{-1}$ and $\textnormal{inj}(p)\ge C^{-1}$. 
\end{lem}

\begin{proof}
We shall use $C$ to denote all positive constants, whose value may vary from in lines.
Let $\gamma:[0,\infty)\ri \Sigma$ be a unit speed geodesic emanating from $p$ in $\Sigma$ such that $\gamma$ is contained in $N$. Then by the curvature assumption and the Jacobi comparison we get $\varphi(\gamma(1))\ge c:=\sin 1>0$. Since $\varphi(\gamma(s))$ increases in $s$ and $d(p,\gamma(1))\le 1$, we can find $s_0\ge 1$ such that $d(p,\gamma(s_0))=1$ and $\varphi(\gamma(s_0))\ge c$.

Let $q=\gamma(s_0)$. We claim $d(q,\Gamma)\ge c$:
First, suppose $d(q,\Gamma)=d(q,x)$ for some $x\in\Gamma$. Let $\sigma$ be the unit speed minimizing geodesic from $x$ and $q$, then by the first variation formula we see that $\sigma$ intersects with $\Gamma$ orthogonally at $x$.
Consider all the preimages of $\sigma$ under the Riemannian submersion $M\ri N$, which
form a smooth submanifold with induced metric $dr^2+\varphi^2(\sigma(r))g_{S^{n-2}}$.
Then by the vertical tangent condition at $x$, we have $\left.\frac{d}{dr}\right|_{r=0}\varphi(\sigma(r))=1$, which by $\varphi(q)\ge c$ and the concavity of $\varphi$ implies $d(q,\Gamma)\ge c$.

Choose some $y\in\Gamma$ such that $d(p,y)=1$. Let $pq,yp,yq$ be minimizing geodesics between these points. By replacing $py$ with its image under a suitable $\mathbb{Z}_2\times O(n-1)$-action, we may assume $\measuredangle ypq\le\frac{\pi}{2}$.
So by angle comparison we get $\widetilde{\measuredangle} ypq\le\measuredangle ypq\le\frac{\pi}{2}$, and hence $\widetilde{\measuredangle} yqp\ge \frac{\pi}{4}$ since $d(p,y)=d(p,q)$. 
So for some $y'\in yq$, $p'\in pq$ such that $d(y',q)=d(p',q)= c$ we have $\widetilde{\measuredangle} y'qp'\ge \widetilde{\measuredangle} yqp\ge\frac{\pi}{4}$, and hence $d(y',p')\ge C^{-1}$, which by volume comparison implies $vol_N(B_N(q,c))\ge C^{-1}$.
Moreover, since $\varphi$ is concave and $\varphi(q)=1$, we have $\varphi\ge C^{-1}$ on $B_N(q,\frac{c}{2})$, integrating it on $B_N(q,c)$ it implies $vol_M(B_M(p,1))\ge C^{-1}$.
The assertion about the injectivity radius now follows from the volume lower bound and the curvature bound $R\le R(p)=1$.
\end{proof}

Recall that if $(M^n,g,f,p)$ is an expanding gradient soliton satisfying 
\begin{equation}
    \Ric+\lambda g=\nabla^2 f
\end{equation}
for some $\lambda>0$. Then it generates a Ricci flow $g(t):=(2\lambda t)\phi^*_{t-\frac{1}{2\lambda}}g$, $t\in(0,\infty)$, where $\{\phi_s\}_{s\in\left(-\frac{1}{2\lambda},\infty\right)}$ is the one-parameter diffeomorphisms generated by the time-dependent vector field $\frac{-1}{1+2\lambda s}\nabla f$ with $\phi_0$ the identity. Moreover, $g(t)$ is an expanding gradient soliton  satisfying
\begin{equation}
    \Ric(g(t))+\frac{1}{2t}g(t)=\nabla^2 f_t,
\end{equation}
where $f_t=\phi^*_{t-\frac{1}{2\lambda}}f$.

Let $(M^n_i,g_i,f_i,p_i)$ be a sequence of $\mathbb{Z}_2\times O(n-1)$-symmetric expanding gradient solitons with positive curvature operator, which satisfies $R(p_i)=1$ and the asymptotic volume ratio $\AVR(g_i)\ri0$ as $i\rii$.
Let $C_i>0$ be the constant such that $(M^n_i,g_i,f_i,p_i)$ satisfies the soliton equation
\begin{equation}
    \Ric(g_i)+\frac{1}{2C_i}g_i=\nabla^2f_i.
\end{equation}
Then the following lemma shows $C_i\rii$ as $i\rii$, and hence there is a subsequence of $(M^n_i,g_i,f_i,p_i)$ smoothly converging to a steady gradient soliton.

\begin{lem}\label{l: expanding converging to steady}
Let $(M^n_i,g_i,f_i,p_i)$ be a sequence of $\mathbb{Z}_2\times O(n-1)$-symmetric expanding gradient solitons with positive curvature operator. Suppose $R_{g_i}(p_i)=1$ and $\AVR(g_i)\rightarrow0$ as $i\rightarrow\infty$. Then a subsequence of $(M_i,g_i,f_i,p_i)$ smoothly converges to an n-dimensional $\mathbb{Z}_2\times O(n-1)$-symmetric steady gradient soliton $(M,g,f,p)$.
\end{lem}

\begin{proof}
Suppose $(M^n_i,g_i,f_i,p_i)$ satisfies 
\begin{equation}\label{e: g_i}
    \Ric(g_i)+\frac{1}{2C_i}g_i=\nabla^2f_i
\end{equation}
for some constant $C_i>0$.
Let $(M_i,\widetilde{g}_i(t),f_{i,t},p_i)$, $t\in(0,\infty)$, be the Ricci flow generated by $(M_i,g_i,f_i,p_i)$, where $\widetilde{g}_i(t)=\frac{t}{C_i}\phi^*_{i,t-C_i}g_i$, $f_{i,t}=\phi^*_{i,t-C_i} f_i$, and $\{\phi_{i,s}\}_{s\in(-C_i,\infty)}$ is the family of diffeomorphisms generated by $\frac{-s}{s+C_i}\nabla f_i$ with $\phi_0$ the identity. By a direct computation we can show
\begin{equation}
    \Ric(\widetilde{g}_i(t))+\frac{1}{2t}\widetilde{g}_i(t)=\nabla^2f_{i,t},
\end{equation}
for all positive time $t$. In particular, we have $\widetilde{g}_i(C_i)=g_i$ and $R_{\widetilde{g}_i(1)}(p_i)=C_i$.

We claim that $C_i\rightarrow\infty$ as $i\rightarrow\infty$:
Suppose this is not true. Then by passing to a subsequence we may assume $C_i\le C$ for some constant $C>0$ and all $i$. We shall use $C$ to denote all positive constant that is independent of $i$.

First, by Lemma \ref{l: v_0} we have $\textnormal{inj}_{\widetilde{g}_i(1)}(p_i)\ge C^{-1}$ and 
\begin{equation}
    R_{\widetilde{g}_i(t)}(x)\le R_{\widetilde{g}_i(t)}(p_i)\le \frac{C}{t},
\end{equation}
for all $x\in M_i$ and $t\in(0,\infty)$.
So by Hamilton's compactness for Ricci flow we may assume after passing to a subsequence that $(M_i,\widetilde{g}_i(t),p_i)$, $t\in(0,\infty)$, converges to a smooth
Ricci flow $(M_{\infty},g_{\infty}(t),p_{\infty})$ on $(0,\infty)$.
Assume $f_{i,1}(p_i)=0$, then by $|\nabla f_{i,1}|(p_i)=0$ and $\Ric_{\widetilde{g}_i(1)}+\frac{1}{2}\widetilde{g}_i(1)=\nabla^2f_{i.1}$, we can apply Shi's derivative estimates to get bounds for higher derivatives of curvatures, and thus bounds for higher derivatives of $f_{i,1}$. So we may assume $f_{i,1}$ converges to a smooth function $f_{\infty}$ satisfying $\Ric_{g_{\infty}(1)}+\frac{1}{2}g_{\infty}(1)=\nabla^2 f_{\infty}$, which makes $(M_{\infty},g_{\infty}(t),p_{\infty})$ an expanding gradient soliton.  
Since $R_{\widetilde{g}_i(t)}\le\frac{C}{t}$, it follows that $R_{g_{\infty}(t)}\le\frac{C}{t}$.

This curvature condition combined with Hamilton's distance distortion estimate gives us a uniform double side control on $d_{\widetilde{g}_i(t)}$ and $d_{g_{\infty}(t)}$, which implies the following pointed Gromov-Hausdorff convergences
\begin{equation}
(M_i,\widetilde{g}_i(t),p_i)\xrightarrow[t\searrow0]{pGH}(C(X_i),o_i),\quad
    (M_{\infty},g_{\infty}(t),p_{\infty})\xrightarrow[t\searrow0]{pGH}(C(X),o),\;\;
\end{equation}
where $X_i,X$ are some compact length spaces, and $o_i,o$ are the cone points of the metric cones $C(X_i),C(X)$. In particular, the first convergence is uniform for all $i$, which implies $(C(X_i),o_i)\xrightarrow[i\rii]{pGH}(C(X),o)$.

Let $\mathcal{H}_n(\cdot)$ denote the n-dimensional Hausdorff measure. Then since it is weakly continuous under the Gromov-Hausdorff convergence \cite{BGP}, we have
\begin{equation}
    \mathcal{H}_n(B(o,1))=\lim_{i\rightarrow\infty}vol(B(o_i,1))=\lim_{i\rightarrow\infty}\AVR(C(X_i))=\lim_{i\rightarrow\infty}\AVR(g_i)=0.
\end{equation}
However, since $(M_{\infty},g_{\infty})$ is an expanding gradient soliton with $\Ric\ge0$, it must have positive asymptotic volume ratio by a result of Hamilton \cite[Prop 9.46]{HaRF}.
So by volume comparison we have
\begin{equation}
    \mathcal{H}_n(B(o,1))=\lim_{t\searrow0}\mathcal{H}_n(B_t(p_{\infty},1))\ge \AVR(g_{\infty}(t))>0,
\end{equation}
a contradiction. This proves the claim at beginning that $C_i\rightarrow\infty$ when $i\rightarrow\infty$.

Let $\widehat{g}_i(t)=\widetilde{g}_i(t+C_i)$, $t\in(-C_i,\infty)$, then $\widehat{g}_i(0)=g_i$,  $R_{\widehat{g}_i(0)}(p_i)=1$, and
\begin{equation}
    R_{\widehat{g}_i(t)}(x)=R_{\widetilde{g}_i(t+C_i)}(x)\le\frac{C_i}{t+C_i}\le 2,
\end{equation}
for all $x\in M_i$ and $t\in(-\frac{C_i}{2},\infty)$.
By Lemma \ref{l: v_0} there is a subsequence of $(M_i,\widehat{g}_i(t),p_i)$ which smoothly converges to a Ricci flow $(M,g(t),p)$, $t\in(-\infty,\infty)$.
Moreover, by the equation \eqref{e: g_i} and Shi's derivative estimates we obtain uniform bounds for all higher derivatives of $f_i$. Since $C_i\rii$ as $i\rii$, we may assume by passing to a subsequence that $f_i$ smoothly converges to a function $f$ on $M$ which satisfies $\Ric(g)=\nabla^2f$. So $(M,g(0),f,p)$ is a steady gradient soliton.
The $\mathbb{Z}_2\times O(n-1)$-symmetry is an easy consequence of the smooth convergence.

\end{proof}

Now we prove Theorem \ref{t: existence with prescribed eigenvalue}. 

\begin{proof}[Proof of Theorem \ref{t: existence with prescribed eigenvalue}]

We claim that there is a sequence of smooth families of $\mathbb{Z}_2\times O(n-1)$-symmetric Riemannian manifolds $\{X_{i,\mu},\mu\in[0,1]\}_{i=0}^{\infty}$ diffeomorphic to $S^{n-1}$, satisfying the following:
\begin{enumerate}
    \item $X_{i,0}$ is a rescaled round (n-1)-sphere;
    \item $\diam(X_{i,1})\rightarrow \pi$ as $i\rii$;
    \item $K(X_{i,\mu})> 1$, where $K$ denotes the sectional curvature;
    \item $\lim_{i\rii}\sup_{\mu\in[0,1]}vol(X_{i,\mu})=0$.
\end{enumerate}
We say $X_{i,\mu}$ is $\mathbb{Z}_2\times O(n-1)$-symmetric if it is rotationally symmetric, and there is a $\mathbb{Z}_2$-isometry that maps the two centers of rotations to each other. 
We prove the claim in dimension $n=3$ below, and the case for $n>3$ follows in the same way.

First, we construct a sequence of smooth $\sy$-symmetric surfaces $\{X_{i,1}\}_{i=1}^{\infty}$ with $K(X_{i,1})>1$, $\diam(X_{i,1})\ri\pi$ and $vol(X_{i,1})\ri0$ as $i\rii$.
For each large $i\in\mathbb{N}$, 
let $g_i$ be the metric of the surface of revolution $(i^{-1}\sin r\cos\theta,i^{-1}\sin r\sin\theta,r)$, $r\in[0,\pi]$ and $\theta\in[0,2\pi]$.
Then by a direct computation we see that $K_{\min}(g_i)=(i^{-2}+1)^{-2}$.
Then by some standard smoothing arguments and suitable rescalings, we obtain the desired sequence $\{X_{i,1}\}_{i=1}^{\infty}$.

Second, for each large $i$, let $h_i(t)$ be the Ricci flow with $h_i(0)=X_{i,1}$, and assume its curvature blows up at $T_i>0$.
Let $K_i(t)$ be the minimum of $K(h_i(t))$, and $V_i(t)$ be the volume with respect to $h_i(t)$. Then we can find
a smooth function $r_i:[0,T_i]\rightarrow\mathbb{R}_+$ such that $r_i(0)=1$, $r_i(t)\le\min\{\sqrt{\frac{K_i(t)}{K_i(0)}},\sqrt{\frac{V_i(0)}{V_i(t)}}\}$ for all $t\in[0,T_i]$, and $r_i(t)=\sqrt{\frac{V_i(0)}{V_i(t)}}$ when $t$ is close to $T_i$ (note $\sqrt{\frac{V_i(0)}{V_i(t)}}<\sqrt{\frac{K_i(t)}{K_i(0)}}$ when $i$ is sufficiently large since $\lim_{i\rii}V_i(0)=0$ and $\limsup_{i\rii}K_i(0)\le1$). 
Then the rescaled Ricci flow $r_i^2(t)h_i(t)$ converges to a smooth round 2-sphere when $t\ri T_i$.
Moreover, by letting $X_{i,\mu}=r_i^2(T_i(1-\mu))h_i(T_i(1-\mu))$, $u\in[0,1]$,
we obtain a smooth family of $\sy$-symmetric surfaces $\{X_{i,u}\}$ with $K(X_{i,\mu})> 1$, $vol(X_{i,\mu})\le  vol(X_{i,1})$, and $X_{i,0}$ is a round $2$-sphere.
So the claim holds.

Therefore, for each fixed $i$, by applying Deruelle's result  \cite[Theorem 1.4]{De15} to $X_{i,\mu}$, $\mu\in[0,1]$, we obtain a smooth family of n-dimensional expanding gradient solitons $(M_{i,\mu},g_{i,\mu},p_{i,\mu}),\mu\in[0,1]$, with positive curvature operator, and asymptotic to $C(X_{i,\mu})$.
Moreover, by \cite[Theorem 1.3]{De15}, the Ricci flow generated by an expanding gradient soliton coming out of $C(X_{i,\mu})$ is unique. 
So any isometry of $C(X_{i,\mu})$ is an isometry at any positive time of the Ricci flow.
In particular, it implies that $(M_{i,\mu},g_{i,\mu},p_{i,\mu})$ is $\mathbb{Z}_2\times O(n-1)$-symmetric and  $(M_{i,0},g_{i,0},p_{i,0})$ is rotationally symmetric.

By some suitable rescalings we may assume $R(p_{i,\mu})=1$, and by item (4) we have $\lim_{i\rii}\sup_{\mu\in[0,1]}\AVR(g_{i,\mu})=\lim_{i\rii}\sup_{\mu\in[0,1]}\AVR(C(X_{i,\mu}))=0$. 
So we can apply Lemma \ref{l: expanding converging to steady} and by passing to a subsequence, we may assume $(M_{i,0},g_{i,0},p_{i,0})$ and $(M_{i,1},g_{i,1},p_{i,1})$ smoothly converge to two steady gradient solitons  $(M_{\infty,0},g_{\infty,0},p_{\infty,0})$ and $(M_{\infty,1},g_{\infty,1},p_{\infty,1})$ respectively.
On the one hand,
since $(M_{i,0},g_{i,0},p_{i,0})$ is rotationally symmetric, it follows that $(M_{\infty,0},g_{\infty,0},p_{\infty,0})$ is rotationally symmetric, and hence is a Bryant soliton, see e.g. \cite{HaRF}. 

On the other hand, since $\diam(X_{i,1})\ri\pi$ when $i\rii$, the asymptotic cone for each $(M_{i,1},g_{i,1},p_{i,1})$ converges to a half-plane, or equivalently a cone over the interval $[0,\pi]$.
So for each $j\in\mathbb{N}$ and all sufficiently large $i$, we can find points $q_{i,j},r_{i,j}\in M_{i,1}$ such that $d(q_{i,j},p_{i,1})=d(r_{i,j},p_{i,1})=j$ and $\widetilde{\measuredangle}q_{i,j}p_{i,1}r_{i,j}\ge \pi-j^{-1}$.
Passing to the limit we obtain points $q_j,r_j\in M_{\infty,1}$ with $d(q_j,p_{\infty,1})=d(r_j,p_{\infty,1})=j$ and $\widetilde{\measuredangle}q_jp_{\infty,1}r_j\ge \pi-j^{-1}$.
Then letting $j\rii$ and passing to a subsequence, the geodesics $p_{\infty,1}q_j,p_{\infty,1}r_j$ converge to two rays which together form a line passing through $p_{\infty,1}$.
Then by the strong maximum principle of Ricci flow, $(M_{\infty,1},g_{\infty,1})$ is the product of $\mathbb{R}$ and an (n-1)-dimensional rotationally symmetric steady gradient soliton with positive curvature operator, which is an (n-1)-dimensional Bryant soliton if $n>3$, and a cigar soliton if $n=3$, see e.g. \cite{HaRF}. 

For a $\mathbb{Z}_2\times O(n-1)$-symmetric expanding or steady gradient soliton $(M,g,p)$ with non-negative curvature operator, we write $\lambda_1(g),\lambda_2(g)=\dots=\lambda_n(g)$ to be the $n$ eigenvalues of the Ricci curvature at $p$ in the directions of $\Gamma'(0)$ and its orthogonal complement subspace $ T_{p}\Sigma=(\Gamma'(0))^{\perp}$.
For any $\alpha\in(0,1)$, since $\frac{\lambda_1}{\lambda_2}(g_{\infty,0})=1$ and $\frac{\lambda_1}{\lambda_2}(g_{\infty,1})=0$,
we have $\frac{\lambda_1}{\lambda_2}(g_{i,0})>\alpha$ and $\frac{\lambda_1}{\lambda_2}(g_{i,1})<\alpha$ when $i$ is sufficiently large.
Since $\frac{\lambda_1}{\lambda_2}(g_{i,\mu})$ is a continuous function of $\mu$ for each fixed $i$, there is some $\mu_i\in(0,1)$ such that $\frac{\lambda_1}{\lambda_2}(g_{i,\mu_i})=\alpha$.
Applying Lemma \ref{l: expanding converging to steady} to the sequence $(M_{i,\mu_i},g_{i,\mu_i},p_{i,\mu_i})$ and taking a limit, we obtain an n-dimensional $\mathbb{Z}_2\times O(n-1)$-symmetric steady gradient soliton $(M,g,p)$ with $\frac{\lambda_1}{\lambda_2}(g)=\alpha$. This proves Theorem \ref{t: existence with prescribed eigenvalue}.
\end{proof}

\end{section}

\begin{section}{Asymptotic geometry of steady gradient solitons}\label{s: construction}
In this section, we study the asymptotic geometry of n-dimensional $\mathbb{Z}_2\times O(n-1)$-symmetric steady gradient solitons.
We show that such a soliton strongly dimension reduces along an edge to an $(n-1)$-dimensional ancient Ricci flow (see below for definitions).
In particular, when $n=3$, the 2d ancient Ricci flow is the cigar soliton, assuming in additional that the scalar curvature does not vanish at infinity.
See also \cite{BDM} for discussions of dimension reductions of 4d non-collapsed steady gradient solitons.

\begin{defn}
Let $(M^n,g,p)$ be an $n$-dimensional $\mathbb{Z}_2\times O(n-1)$-symmetric steady gradient soliton.
We say that it \textbf{strongly dimension reduces} along $\Gamma$
to an $(n-1)$-dimensional ancient Ricci flow $(N,g(t))$, if for any sequence $s_i\rii$, a subsequence of $(M,K_ig(K^{-1}_it),\Gamma(s_i))$, $t\in(-\infty,0]$, where $K_i=R(\Gamma(s_i))$, smoothly converges to the product of $\mathbb{R}$ and $(N,g(t))$. 

We also say an $(n-1)$-dimensional ancient Ricci flow $(N,h(t))$ is a \textbf{\textit{dimension reduction}} of $(M^n,g,p)$ along $\Gamma$, if there exists $s_i\rii$ such that $(M,K_ig(K^{-1}_it),\Gamma(s_i))$, $t\in(-\infty,0]$, where $K_i=R(\Gamma(s_i))$, smoothly converges to the product of $\mathbb{R}$ and $(N,h(t),p_{\infty})$.
\end{defn}


First we prove a lemma about the relations between the potential function and distance function.


\begin{lem}\label{l: f and d}
Let $(M^n,g,f,p)$ be an n-dimensional steady gradient soliton with positive curvature operator. Suppose $\gamma:(0,\infty)\ri M$ is an integral curve of $\frac{\nabla f}{|\nabla f|}$, and $\lim_{s\ri0}\gamma(s)=p$. Then for any $\epsilon>0$, there exists $s_0>0$ such that for any $s_1>s_2>s_0$ we have
\begin{equation}\label{e: d and s}
    (1-\epsilon)(s_2-s_1)\le d(\gamma(s_1),\gamma(s_2))\le (s_2-s_1).
\end{equation}
In particular, we have $(1-\epsilon)s\le d(p,\gamma(s))\le s$
for all $s\ge s_0$.
Moreover, let $\sigma$ be a unit speed minimizing geodesic between $p$ and $\sigma(0):=\gamma(s)$.
Then 
\begin{equation}\label{e: angle}
    \measuredangle (\sigma'(0),\nabla f)\le\epsilon.
\end{equation}
\end{lem}
 
\begin{proof}
Without loss of generality, we may assume $f(p)=0$ and $\lim_{s\rii}|\nabla f|(\gamma(s))=1$ after a suitable rescaling.
We use $\epsilon=\epsilon(s)$ to denote all functions such that $\lim_{s\rii}\epsilon(s)=0$. 

On the one hand, for any $s_2>s_1\ge 0$, let $\sigma:[0,D]\ri M$ be a minimizing geodesic from $\gamma(s_1)$ to $\gamma(s_2)$, where $D=d(\gamma(s_1),\gamma(s_2))$. 
Since $\frac{d}{dr}\langle\nabla f,\sigma'(r)\rangle=\nabla^2 f(\sigma'(r),\sigma'(r))\ge 0$, we obtain
\begin{equation}\label{e: d and f}
    f(\gamma(s_2))-f(\gamma(s_1))=\int_{0}^{D}\langle\nabla f,\sigma'(r)\rangle\,dr\le D\,\langle\nabla f,\sigma'(D)\rangle,
\end{equation}
which by $|\nabla f|\le 1$ implies
\begin{equation}
    f(\gamma(s_2))-f(\gamma(s_1))\le d(\gamma(s_1),\gamma(s_2)).
\end{equation}

On the other hand, since $\lim_{s\rii}|\nabla f|(\gamma(s))=1$,
there is $s_0>0$ such that $|\nabla f|(\gamma(s))>1-\epsilon$ for all $s\ge s_0$. Therefore, for all $s_2>s_1\ge s_0$ we have
\begin{equation}\label{e: f-f}
    f(\gamma(s_2))-f(\gamma(s_1))=\int_{s_1}^{s_2}\langle \nabla f,\gamma'(r)\rangle\,dr=\int_{s_1}^{s_2}|\nabla f|(\gamma(r))\,dr\ge (1-\epsilon)(s_2-s_1),
\end{equation}
which together with \eqref{e: d and f} proves the first inequality in \eqref{e: d and s}, where the second inequality is an easy consequence of $|\gamma'(s)|=1$.
The inequality of $d(p,\gamma(s))$ follows \eqref{e: d and s} and a triangle inequality.

Now let $\sigma:[0,d(p,\gamma(s))]\ri M$ be a minimizing geodesic from $p$ to $\gamma(s)$.
Then \eqref{e: d and f} implies 
\begin{equation}\label{e: d and f'}
    f(\gamma(s))\le d(p,\gamma(s))\,\langle\nabla f,\sigma'(d(p,\gamma(s)))\rangle.
\end{equation}
Moreover, by
$\eqref{e: f-f}$ and $\lim_{s\rii}f(\gamma(s))=\infty$ we have
\begin{equation}
    d(\gamma(s_0),\gamma(s))\le s-s_0\le(1+\epsilon)(f(\gamma(s))-f(\gamma(s_0)))\le(1+\epsilon)f(\gamma(s))
\end{equation}
for all $s$ sufficiently large, which  by triangle inequality and $\lim_{s\rii}d(p,\gamma(s))=\infty$ implies
\begin{equation}
    \begin{split}
        d(p,\gamma(s))\le d(p,\gamma(s_0))+d(\gamma(s_0),\gamma(s))
        \le(1+\epsilon)d(\gamma(s_0),\gamma(s))
        \le(1+\epsilon)f(\gamma(s)).
    \end{split}
\end{equation}
This combining with \eqref{e: d and f'} and $|\nabla f|\le 1$ yields
\begin{equation}
    \left\langle\frac{\nabla f}{|\nabla f|},\sigma'(d(p,\gamma(s)))\right\rangle\ge
    \frac{f(\gamma(s))}{d(p,\gamma(s))\,|\nabla f|}\ge
    1-\epsilon,
\end{equation}
which proves \eqref{e: angle}.
\end{proof}


The following lemma shows that all dilation sequence along $\Gamma$ smoothly converges to a limit after passing to a subsequence.
The limits are all products of a line and some rotationally symmetric ancient solution.

Our main tool is Perelman's curvature estimate for Ricci flows with non-negative
curvature operator, see for example \cite[Corollary 45.1(b)]{KL}, or a more general result in \cite[Proposition 3.2]{BamA}.
It implies that for a Ricci flow with non-negative
curvature operator $(M,g(t))$, $t\in[-1,0]$, assume $B_{g(0)}(x_0,1)\ge\kappa>0$ for some $x_0\in M$, then there is $C(\kappa)>0$ such that $R(x_0,0)\le C$.

\begin{lem}\label{l: curvature bound and inj bound}
Let $(M^n,g,p)$ be a non-flat $\mathbb{Z}_2\times O(n-1)$-symmetric $n$-dimensional steady gradient soliton. Then there is $C>0$ such that the following holds:

For any $s_i\ri+\infty$, a subsequence of $(M,K_ig(K^{-1}_it),\Gamma(s_i))$, $t\in(-\infty,0]$, $K_i=R(\Gamma(s_i))$, smoothly converges to an ancient Ricci flow $(\mathbb{R}\times g_{\infty}(t),p_{\infty})$, where $g_{\infty}(t)$ is an $(n-1)$-dimensional ancient Ricci flow with positive curvature operator and $R\le C$. 
Moreover,   $R^{-1/2}(\Gamma(s_i))\Gamma'(s_i)$ smoothly converges to a unit vector in the $\mathbb{R}$-direction of $\mathbb{R}\times g_{\infty}(t)$, and $g_{\infty}(t)$ is rotationally symmetric around $p_{\infty}$.  
\end{lem}

\begin{proof}
If $\Rm>0$ does not hold, then by the strong maximum principle the soliton is $\mathbb{R}\times\textnormal{Bryant}$ for $n\ge4$, or $\mathbb{R}\times\cigar$ for $n=3$.
The conclusion clearly holds in these cases, so we may assume $\Rm>0$.

Let $r(s)=\sup\{\rho>0: vol(B(\gs,\rho))\ge\frac{\omega}{2}\rho^n\}$ where $\omega$ is the volume of the unit ball in the Euclidean space $\mathbb{R}^n$.
Since the asymptotic volume ratio of any non-flat ancient Ricci flow with non-negative curvature operator is zero by \cite[Corollary 45.1(b)]{KL},
we have $r(s)<\infty$ for each $s$, and $\lim_{s\rii}\frac{r(s)}{s}=0$. Moreover, by the choice of $r(s)$ we have $vol(B(\gs,r(s)))=\frac{\omega}{2} r^n(s)$. 

For any $D>0$ and any $x\in B(\gs,Dr(s))$, by the volume comparison we have $vol(B(x,r(s)))\ge C_1^{-1}r^n(s)$ for some $C_1(D)>0$.
Therefore, by \cite[Corollary 45.1(b)]{KL} we can find constants $C_2(D)>0$ such that $R\le C_2r^{-2}(s)$ in $B(\gs,Dr(s))$.
By Hamilton's Harnack inequality $\frac{d}{dt}R(\cdot,t)\ge0$ for ancient complete Ricci flow with non-negative curvature operator \cite{Harnack}, this implies $R(x,t)\le C_2r^{-2}(s)$ for all $x\in B(\gs,Dr(s))$ and $t\in(-\infty,0]$.
In particular, there is $C_0>0$ such that $C_0^{-1}r(s)\le R^{-1/2}(\Gamma(s))$, and $\textnormal{inj}(\gs)\ge C_0^{-1}r(s)$ by the volume bound.

Therefore, for any $s_i\rii$, by Shi's derivative estimates and Hamilton's compactness theorem for Ricci flow, a subsequence of $(M,r^{-2}(s_i)g(r^2(s_i)t),\Gamma(s_i))$, $t\in(-\infty,0]$, converges to an ancient solution $h_{\infty}(t)$. Let $\Gamma_i(s)=\Gamma(r(s_i)s+s_i)$, $s\in(-\infty,\infty)$.
Suppose $\Gamma_i$ converges to the geodesic $\Gamma_{\infty}$ in $h_{\infty}(0)$ as $i\rii$, modulo the diffeomorphisms.
We claim that $\Gamma_{\infty}$ is a line:
Since $\lim_{s\rii}\frac{r(s)}{s}\ri0$, we have $s_i-Dr(s_i)\rii$, by which we can apply Lemma \ref{l: f and d} and deduce that for any $D>0$ that $\widetilde{\measuredangle}\Gamma_i(-D)\Gamma_i(0)\Gamma_i(D)\ri\pi$ as $i\rii$. 
So $d(\Gamma_{\infty}(-D),\Gamma_{\infty}(D))=2D$. Letting $D\rii$, this implies $\Gamma_{\infty}$ is a line.

Next we claim that there is some $C_3>0$ such that $R^{-1/2}(\Gamma(s))\le C_3r(s)$ for all large $s$. 
Suppose by contradiction this does not hold, then there is a sequence $s_i\rii$ such that $\lim_{i\rii}\frac{R^{-1/2}(\Gamma(s_i))}{r(s_i)}=0$.
Then by taking a subsequence we may assume $(r^{-2}(s_i)g,\Gamma(s_i))$ converges to $(\mathbb{R}\times g_{\infty}(t),p_{\infty})$, where $g_{\infty}(t)$ is some $(n-1)$-dimensional ancient solution. 

On the one hand, as a consequence of taking the limit, we have $vol(B(p_{\infty},1))=\frac{\omega}{2}$ and $R(p_{\infty})=0$, which by the strong maximum principle implies that $g_{\infty}(t)$ is flat.
On the other hand,
since $\Gamma_i$ converges to a line, we can find a sequence $D_i\rii$ such that $\Sigma_i:=\exp_{\Gamma(s_i)}(\Gamma'(s_i)^{\perp})
\cap B(\Gamma(s_i),D_ir(s_i))$ with the metric $g_{\Sigma_i}$ induced by $g$ is a smooth surface which is rotationally symmetric around $\Gamma(s_i)$, and $(r^{-2}(s_i)g_{\Sigma_{i}},\Gamma(s_i))$ smoothly converges to $(g_{\infty}(0),p_{\infty})$.
So $g_{\infty}(0)$ is rotationally symmetric around $p_{\infty}$.
Since $g_{\infty}(0)$ is flat, it must be isometric to $\mathbb{R}^{n-1}$, which implies $vol(B(p_{\infty},1))=\omega>\frac{\omega}{2}$, a contradiction.

Then it follows from $C_0^{-1}r(s)\le R^{-1/2}(\gs)\le C_3r(s)$ that $(M,K_ig(K^{-1}_it),\Gamma(s_i))$, $t\in(-\infty,0]$, $K_i=R(\Gamma(s_i))$, smoothly converges to an ancient Ricci flow $(\mathbb{R}\times g_{\infty}(t),p_{\infty})$ as claimed.
Since $g_{\infty}(t)$ is rotationally symmetric and has positive curvature, the uniform curvature bound $R\le C$ follows easily by applying \cite[Corollary 45.1(b)]{KL}.
\end{proof}

As a corollary of Lemma \ref{l: curvature bound and inj bound}, we show that the n-dimensional steady gradient solitons from Theorem \ref{t: existence with prescribed eigenvalue} are all non-collapsed if $n\ge4$. 

\begin{defn}
A Riemannian manifold $(M^n,g)$ is non-collapsed if there exists a constant $\kappa>0$ such that for any $x\in M$ and $r>0$, if $|\Rm|\le r^{-2}$ in the ball $B_g(x,r)$, then $vol_g(B_g(x,r))\ge\kappa r^n$. Otherwise we say $(M,g)$ is collapsed.
\end{defn}

\begin{cor}
For any $n\ge4$, let $(M^n,g,p)$ be an n-dimensional non-flat $\mathbb{Z}_2\times O(n-1)$-symmetric steady gradient soliton. Then it is non-collapsed.
\end{cor}

\begin{proof}
Let $\omega$ be the volume of the unit ball in $\mathbb{R}^n$.
Suppose the conclusion is not true, then there is a sequence of points $x_i\in M$ such that $\frac{r_i}{\overline{r}_i}\rii$ as $i\rii$, where $\overline{r}_i=\sup\{\rho>0: vol_g(B_g(x_i,\rho))\ge\frac{\omega}{2}\rho^n\}$, and $r_i=\sup\{\rho>0: |\Rm|\le\rho^{-2}\textnormal{ in }B_g(x_i,\rho)\}$.
Then by the same limiting argument as Lemma \ref{l: curvature bound and inj bound}, we may assume by passing to a subsequence that $(M,\overline{r}_i^{-2}g,x_i)$ smoothly converges to a manifold $(M_{\infty},g_{\infty},x_{\infty})$, which is flat and satisfies $vol(B(x_{\infty},1))=\frac{\omega}{2}$.

Let $g_i=\overline{r}_i^{-2}g$.
We first assume that there are a constant $C>0$ and $y_i\in\Gamma$ such that $d_{g_i}(x_i,y_i)\le C$ for all $i$.
Then a subsequence of $(M,g_i,y_i)$ converges to $(M_{\infty},g_{\infty},y_{\infty})$ for some $y_{\infty}\in M_{\infty}$. 
By Lemma \ref{l: curvature bound and inj bound}, $(M_{\infty},g_{\infty})$ is a product of $\mathbb{R}$ and an $(n-1)$-dimensional rotationally symmetric manifold. 
Since $(M_{\infty},g_{\infty})$ is flat, it must be isometric to $\mathbb{R}^n$, which contradicts the choice of $\omega$.

Next, assume $\lim_{i\rii}d_{g_i}(x_i,\Gamma)=\infty$.
Let $h_i$ be the metric induced by $g_i$ on the totally geodesic surface $N$, and assume $x_i\in N$.
Then $g_i=h_i+\varphi_i^2g_{S^{n-2}}$ on $B_{g_i}(x_i,\frac{1}{2}d_{g_i}(x_i,\Gamma))$, where $\varphi_i=\overline{r}_i^{-1}\varphi$.
So it follows easily that $vol_{h_i}(B_{h_i}(x_i,1))\ge c(\omega)$ for some $c(\omega)>0$.
Since $B_{h_i}(x_i,\frac{1}{2}d_{g_i}(x_i,\Gamma))$ is relatively compact in $N$, it follows by the same curvature estimates as Lemma \ref{l: curvature bound and inj bound} that a subsequence of $(N,h_i,x_i)$ smoothly converges to a complete manifold $(N_{\infty},h_{\infty},x_{\infty})$, which is diffeomorphic to $\mathbb{R}^2$.
Since $(N_{\infty},h_{\infty})$ is totally geodesic in $(M_{\infty},g_{\infty})$, it is isometric to $\mathbb{R}^2$.

If $\varphi_i(x_i)\rii$ as $i\rii$, it is easy to see that $(M_{\infty},g_{\infty})$ is isometric to $\mathbb{R}^n$, a contradiction.
Otherwise, there is $C>0$ such that $\varphi_i(x_i)\le C$ for all $i$. Then by the curvature estimates and \eqref{e: computation}, a subsequence of $\varphi_i$ smoothly converges to a positive function $\varphi_{\infty}$, such that $g_{\infty}=g_{\mathbb{R}^2}+\varphi_{\infty}^2g_{S^{n-2}}$. Since $n\ge4$, this contradicts the fact that $(M_{\infty},g_{\infty})$ is flat, hence proves the corollary. 

\end{proof}

To rephrase the statement of Lemma \ref{l: curvature bound and inj bound} and use it to prove a more accurate dimension reduction theorem in dimension 3, we introduce the definition of $\epsilon$-closeness between two Ricci flows.
\begin{defn}
For any $\epsilon>0$, we say a pointed Ricci flow $(M_1,g_1(t),p_1)$, $t\in[-T,0]$, is \textbf{\textit{$\epsilon$-close}} to a pointed Ricci flow $(M_2,g_2(t),p_2)$, $t\in[-T,0]$, if there is a diffeomorphism onto its image $\phi:B_{g_2(0)}(p_2,\epsilon^{-1})\ri M_1$, such that $\phi(p_2)=p_1$ and
$\|\phi^*g_1(t)-g_2(t)\|_{C^{[\epsilon^{-1}]}(U)}<\epsilon$ for all $t\in[-\min\{T,\epsilon^{-1}\},0]$, where the norms and derivatives are taken with respect to $g_2(0)$.
\end{defn}

By this definition, Lemma \ref{l: curvature bound and inj bound} shows that $(R(\gs)g(R^{-1}(\gs)t),\gs)$ is $\epsilon$-close to the product of $\mathbb{R}$ and a dimension reduction for all sufficiently large $s$. Moreover, a dimension reduction  $(M_{\infty},g_{\infty}(t),p_{\infty})$ is an $(n-1)$-dimensional ancient solution with positive curvature operator and it is rotationally symmetric around $p_{\infty}$.

In dimension 3, the next theorem shows that $M_{\infty}$ is non-compact, if the original soliton is not a Bryant soliton.
Moreover, if $\lim_{s\rii}R(\gs)>0$, then the soliton strongly dimension reduces along $\Gamma$ to a cigar soliton.

\begin{theorem}\label{l: geometry at infinity}(Dimension Reduction)
Let $(M,g,f,p)$ be a non-flat 3d $\sy$-symmetric steady gradient soliton, which is not a Bryant soliton.
Then any dimension reduction of $(M,g,p)$ along $\Gamma$ is non-compact.
In particular, if $\lim_{s\rii}R(\gs)>0$, then $(M,g,p)$ strongly dimension reduces along $\Gamma$ to a cigar soliton $(M_{\infty},g_{\infty}(t),p_{\infty})$, $t\in(-\infty,0]$, with $R(p_{\infty},0)=1$.
\end{theorem}

\begin{proof}
Let $\epsilon>0$ be sufficiently small.
We denote by $\epsilon_{\#}$ all positive constants that depend on $\epsilon$ such that $\epsilon_{\#}\ri0$ as $\epsilon\ri0$.

For each sufficiently large $s$, by Lemma \ref{l: curvature bound and inj bound} there is a dimension reduction $(h_s(t),p_s)$ of $(M,g,p)$ along $\Gamma$, such that $(R(\gs)g(R^{-1}(\gs)t),\gs)$ is $\epsilon$-close
to $(\mathbb{R}\times h_s(t),p_s)$.
By Lemma \ref{l: curvature bound and inj bound}, $(h_s(t),p_s)$ is a 2d ancient Ricci flow rotationally symmetric around $p_s$ and $R(p_s,0)=1$.
Note the choice of $h_s(t)$ may not be unique for a fixed $s$, but any two such solutions are $\epsilon_{\#}$-close to each other.
Let 
\begin{equation}
    F(s)=\diam(h_s(0))\in(0,\infty].
\end{equation}

First, if $\limsup_{s\rii} F(s)<\frac{1}{100\epsilon}$,
then there is $\kappa=\kappa(\epsilon)>0$ such that all $h_s(0)$ is $\kappa$-non-collapsed. This implies easily that $(M,g,p)$ is $\kappa$-non-collapsed, and hence is a Bryant soliton, as a consequence of the uniqueness of the Bryant soliton among 3d non-collapsed steady gradient solitons \cite{brendlesteady3d}, or among 3d $\kappa$-solutions \cite{BK3,classification}. This is a contradiction.
So $\limsup_{s\rii} F(s)\ge\frac{1}{100\epsilon}>100\pi$.

Next, we claim that $F(s)\ge D:=\frac{1}{1000\epsilon}$ for all large $s$:
First, choose $s_0$ such that $F(s_0)\ge 3D$, and let
\begin{equation}
    s_1=\sup\{s\ge s_0\,\mid\,F(\mu)\ge 2D \textnormal{ for all }\mu\in[s,s_0]\}.
\end{equation}
Then $F(s_1)\in[2(1-\epsilon_{\#})D,2(1+\epsilon_{\#})D]$ and $(h_{s_1}(t),p_{s_1})$ is a Rosenau solution by the classification of compact ancient 2d Ricci flows \cite{2dancientcompact}. 
Moreover, assume $\epsilon$ is sufficiently small, then $1-\epsilon_{\#}\le R(p_{s_1},t)\le 1$ for all $t\le 0$, see e.g. \cite[Chap 4.4]{HaRF}, and 
\begin{equation}\label{e: D1}
    \diam(h_{s_1}(t))R^{1/2}(p_{s_1},t)
    \ge (1-\epsilon_{\#})F(s_1)\ge 2(1-\epsilon_{\#}) D
\end{equation}
for all $t\le 0$. 
Moreover, by a distance distortion estimate, see e.g. \cite[Lem 27.8]{KL}, we can find a $t_1\in[-\epsilon^{-1},0)$ such that
\begin{equation}\label{e: D2}
    \diam(h_{s_1}(t_1))R^{1/2}(p_{s_1},t_1)= 4D.
\end{equation}

Since $g(t)=\phi_t^*(g)$, where $\{\phi_t\}_{t\in(-\infty,\infty)}$ is the flow of $-\nabla f$ with $\phi_0$ the identity. We see that
$(g(t),\Gamma(s))$ is isometric to $(g,\phi_t(\Gamma(s)))$, and since
$\Gamma$ is the integral curve of $\frac{\nabla f}{|\nabla f|}$, 
by a direct computation we obtain 
\begin{equation}
    \phi_t(\Gamma(s))=\Gamma\left(s-\int_0^t|\nabla f|(\phi_{\mu}(\Gamma(s)))\,d\mu\right).
\end{equation}
Let $s_2=s_1-\int_0^{T_1}|\nabla f|(\phi_{\mu}(\Gamma(s_1)))\,d\mu$, where $T_1=t_1R^{-1}(\Gamma(s_1))<0$. Then $s_2>s_1$, $\phi_{T_1}(\Gamma(s_1))=\Gamma(s_2)$, and $(g(T_1),\Gamma(s_1))$ is isometric to $(g,\Gamma(s_2))$. The conditions \eqref{e: D1}\eqref{e: D2} imply $F(s)\ge 2(1-\epsilon_{\#})D\ge D$ for all $s\in[s_1,s_2]$, and $F(s_2)\ge 4(1-\epsilon_{\#})D\ge 3D$. In particular, this implies $s_2-s_1\ge R^{-1/2}(\Gamma(s_1))\ge R^{-1/2}(p)$.

Therefore, by induction we find a sequence $\{s_{2k}\}_{k=0}^{\infty}$, such that $s_{2k}-s_{2(k-1)}\ge R^{-1/2}(p)$ for all $k\ge1$ and
\begin{equation}
F(s)\ge D \textnormal{ for all }s\in[s_{2(k-1)},s_{2k}],\quad F(s_{2k})\ge 3D.
\end{equation}
This implies $F(s)\ge D=\frac{1}{1000\epsilon}$ for all large $s$.
Letting $\epsilon\ri0$, it follows that any dimension reduction along $\Gamma$ is non-compact.

Now assume $\lim_{s\rii}R(\Gamma(s))>0$.
Suppose $(g_{\infty}(t),p_{\infty})$ is a dimension reduction, and $(M,R(\Gamma(s_i))g(R^{-1}(\Gamma(s_i))t),\Gamma(s_i))$ smoothly converges to $(M_{\infty},\mathbb{R}\times g_{\infty}(t),p_{\infty})$ for a sequence $s_i\rii$.
Let $f_i=f-f(\Gamma(s_i))$. Then  $f_i$ smoothly converges to a function $f_{\infty}$ on $M_{\infty}$ satisfying $\Ric=\nabla^2 f_{\infty}$ with respect to the metric $\mathbb{R}\times g_{\infty}(0)$. So $g_{\infty}(0)$ is a 2d non-flat steady gradient soliton, which must be a cigar soliton \cite{cigar}. 

\end{proof}

\end{section}

\begin{section}{Existence of 3d flying wings}\label{s: proof}
In this section, we prove 
Theorem \ref{t: a ray implies Bryant soliton}, \ref{l: positive angle} and all the corollaries. 
The asymptotic cone of a 3d $\sy$-symmetric steady gradient soliton is a metric cone over $[-\frac{\alpha}{2},\frac{\alpha}{2}]$ for some $\alpha\in[0,\pi]$ (see Lemma \ref{l: angles between Gamma and gamma}). 
Theorem \ref{t: a ray implies Bryant soliton} shows that the soliton must be a Bryant soliton, if the asymptotic cone is a ray. 
So the family of 3d steady gradient solitons from Theorem \ref{t: existence with prescribed eigenvalue} are all flying wings, which confirms Hamilton's conjecture.
 
Throughout this section we assume
$(M,g,p)$ is a non-flat $\mathbb{Z}_2\times O(2)$-symmetric 3d steady gradient soliton, and $\Gamma$ and $\Sigma$ are the fixed point sets of the $O(2)$ and $\mathbb{Z}_2$-action respectively.

The next lemma shows that the integral of scalar curvature in metric balls increases at least linearly in radius. We remark that this is also a consequence of \cite{Catino}, which shows that the only 3d steady gradient solitons satisfying $\liminf_{s\rii} \frac{1}{s}\int_{B(p,s)}R\,dvol_M=0$ are quotients of $\mathbb{R}^3$ and $\mathbb{R}\times\cigar$.
The proof below is self-contained and more direct under the symmetric assumption.

\begin{lem}\label{l: int of R}
There exists $C>0$ such that $\int_{B(p,s)}R\;dvol_M\ge C^{-1}s$ for sufficiently large $s$.
\end{lem}

\begin{proof}
Fix some small $\epsilon>0$ and let $s_0>0$ be large enough such that Lemma \ref{l: f and d} holds for $\epsilon$.
Consider the covering of $\Gamma([s_0,s])$ by $\{\Gamma([\mu-R^{-1/2}(\Gamma(\mu)),\mu+R^{-1/2}(\Gamma(\mu))])\}_{\mu\in[s_0,s]}$.
Let $\{\Gamma([\mu_i-R^{-1/2}(\Gamma(\mu_i)),\mu_i+R^{-1/2}(\Gamma(\mu_i))])\}_{i=1}^m$ be a Vitali covering of it, which is disjoint from each other and $\Gamma([s_0,s])$ is covered by $\{\Gamma([\mu_i-5R^{-1/2}(\Gamma(\mu_i)),\mu_i+5R^{-1/2}(\Gamma(\mu_i))])\}_{i=1}^m$.
So for any $\mu_i<\mu_j$,
\begin{equation}\label{e: mu}
    \mu_j-\mu_i\ge R^{-1/2}(\Gamma(\mu_i))+R^{-1/2}(\Gamma(\mu_j))\ge R^{-1/2}(\Gamma(\mu_j)),
\end{equation}
and
\begin{equation}\label{e: s-s_0}
    s-s_0\le \sum_{i=1}^m 10R^{-1/2}(\Gamma(\mu_i)).
\end{equation}

Let $c=\frac{1-\epsilon}{4}$, we claim that $B(\Gamma(\mu_i),cR^{-1/2}(\Gamma(\mu_i)))$ and $B(\Gamma(\mu_j),cR^{-1/2}(\Gamma(\mu_j)))$ are disjoint:
Suppose not, then $d(\Gamma(\mu_i),\Gamma(\mu_j))<2cR^{-1/2}(\Gamma(\mu_j))$, and by Lemma \ref{l: f and d} we get
\begin{equation}
    \mu_j-\mu_i\le (1-\epsilon)^{-1}d(\Gamma(\mu_i),\Gamma(\mu_j))\le 2(1-\epsilon)^{-1}c\,R^{-1/2}(\Gamma(\mu_j))<R^{-1/2}(\Gamma(\mu_j)),
\end{equation}
which contradicts \eqref{e: mu}.

By   Theorem \ref{l: geometry at infinity} and Shi's derivative estimates, there is some $C_1>0$ such that
\begin{equation}\label{e: mu_i}
    \int_{B(\gs,cR^{-1/2}(\gs))}R\,dvol_M\ge C_1^{-1}R^{-1/2}(\gs).
\end{equation}
Since $\lim_{s\rii}\frac{R^{-1/2}(\gs)}{s}=0$, which can be seen from the proof of Lemma \ref{l: curvature bound and inj bound}, we have
$B(\Gamma(\mu_i),cR^{-1/2}(\Gamma(\mu_i)))\subset B(p,2s)$ for all $i$. Therefore, by \eqref{e: s-s_0} and \eqref{e: mu_i} we obtain
\begin{equation}
    \int_{B(p,2s)}R\,dvol_M\ge\sum_{i=1}^m\int_{B(\Gamma(\mu_i),cR^{-1/2}(\Gamma(\mu_i)))} R\,dvol_M\ge C_2^{-1}s
\end{equation}
for some $C_2>0$.
\end{proof}

The next lemma shows that for any non-flat $\sy$-symmetric 3d steady gradient soliton $(M,g,p)$, the space of equivalent classes of rays is an interval $[-\frac{\alpha}{2},\frac{\alpha}{2}]$, where $\alpha\in[0,\pi]$. So the asymptotic cone is a sector with angle $\alpha\in[0,\pi]$. Moreover, the minimizing
geodesics between $p$ and points going to infinity along $\Gamma$ and $\Sigma$ converge to a ray in the class $\pm\frac{\alpha}{2}$ and $0$ respectively.

\begin{lem}\label{l: angles between Gamma and gamma}
The asymptotic cone of $(M,g,p)$ is a metric cone $C(X)$ over the interval $X=[-\frac{\alpha}{2},\frac{\alpha}{2}]$ for some $\alpha\in[0,\pi]$, and
\begin{enumerate}
    \item For any sequence $s_i\ri+\infty$, the geodesics between $p$ and $\Gamma(s_i)$ converge to the equivalent class $\frac{\alpha}{2}\in X$. 
    \item For any sequence $q_i\in\Sigma$ and $q_i\rii$, the geodesics between $p$ and $q_i$ converge to the equivalent class $0\in X$.
    \item For any $q_i\in\Sigma$, $q_i\rii$, and $o_i=\Gamma(s_i)$, $s_i\rii$, with $C^{-1}\,d(p,o_i)\le d(p,q_i)\le C\,d(p,o_i)$, we have
    $\lim_{i\rii}\widetilde{\measuredangle}q_ipo_i=\frac{\alpha}{2}$.
\end{enumerate}
\end{lem}

\begin{proof}
The conclusion clearly holds for $\mathbb{R}\times\cigar$ with $\alpha=\pi$, so we may assume $(M,g,p)$ has positive sectional curvature.
For any $s_i\rii$, let $p_i=\Gamma(s_i)$ and $\overline{p}_i=\Gamma(-s_i)$.
Assume after passing to a subsequence that the minimizing geodesics $pp_i,p\overline{p}_i$ converge to rays $\gamma_1,\overline{\gamma}_1$ respectively.
Let $(X,d_X)$ be the space of the equivalent classes of rays, and $\gamma_2,\overline{\gamma}_2\in X$.
We claim that $d_X(\gamma_1,\overline{\gamma}_1)> d_X(\gamma_2,\overline{\gamma}_2)$ unless $\{\gamma_1,\overline{\gamma}_1\}=\{\gamma_2,\overline{\gamma}_2\}$.
If the claim holds, it follows that $X=[-\frac{\alpha}{2},\frac{\alpha}{2}]$ for some $\alpha\in[0,\pi]$.

Let $\gamma_i$ be a minimizing geodesic connecting $p_i$ and $\overline{p}_i$, then $d(p,\gamma_i)\rii$ as $i\rii$, because otherwise $\gamma_i$ would converge to a line, which contradicts with $\Rm>0$.
So for large $i$, the two rays $\gamma_2,\overline{\gamma}_2$ intersect with $\sigma$ at $q_i,\overline{q}_i\neq p$ respectively. Assume $d(p_i,q_i)\le d(p_i,\overline{q}_i)$ by passing to a subsequence if necessary.
Then it is easy to see 
\begin{equation}
    \widetilde{\measuredangle} p_ip\overline{p}_i\ge
    \widetilde{\measuredangle} p_ipq_i+
    \widetilde{\measuredangle} q_ip\overline{q}_i+
    \widetilde{\measuredangle} \overline{p}_ip\overline{q}_i,
\end{equation}
which implies the following when $i\rii$
\begin{equation}
    d_X(\gamma_1,\overline{\gamma}_1)\ge d_X(\gamma_1,\gamma_2)+d_X(\gamma_2,\overline{\gamma}_2)+d_X(\overline{\gamma}_1,\overline{\gamma}_2)\ge d_X(\gamma_2,\overline{\gamma}_2).
\end{equation}
In particular, the equalities hold if and only if $d_X(\gamma_1,\gamma_2)=d_X(\overline{\gamma}_1,\overline{\gamma}_2)=0$, which proves the claim.

Assertion (2) follows immediately from the fact that $\Sigma$ is the fixed point set of the $\mathbb{Z}_2$-action. Assertion (3) is a consequence of (1) and (2) and the fact that $C(X)$ is isometric to the Gromov-Hausdorff limit of $(M,\lambda_ig,p)$ for any sequence $\lambda_i\ri0$.
\end{proof}

From now on we fix a minimizing geodesic  $\gamma:[0,\infty)\ri\Sigma$ starting from $p$ such that $\gamma((0,\infty))\subset N$, and two functions $h_1(s)=d(\gamma(s),\Gamma)$ and $h_2(s)=\varphi(\gamma(s))$ that can be thought of as ``dimensions" of the soliton.
For example, we have $h_1(s)\approx s^{1/2}$, $h_2(s)\approx s^{1/2}$ in a Bryant soliton, and $h_1(s)\approx s$, $\lim_{s\rii}h_2(s)<\infty$ in $\mathbb{R}\times\cigar$.
We establish inequalities between these two functions and $R(\gamma(s))$ in the following three lemmas, when $s$ is sufficiently large.

For convenience, in the rest proofs we shall often use $\epsilon(s)$ to denote all functions such that $\lim_{s\rii}\epsilon(s)=0$, and use $C$ to denote all positive constants.

\begin{lem}\label{l: key estimate}
There exists $C>0$ such that $h_1^2(s)R(\gamma(s))\le C$ for all large $s$.
\end{lem}

\begin{proof}
Without loss of generality we may assume $\alpha<\pi$, because otherwise $(M,g,p)$ is $\mathbb{R}\times\cigar$, where the assertion follows from the exponential decay of the scalar curvature.

Let $p_1=\gamma(s)$ and $p_2=\Gamma(s\,\cos\frac{\alpha}{2})$.
On the one hand, since $\alpha<\pi$, we have $\Gamma(s\,\cos\frac{\alpha}{2})\rii$ as $s\rii$, which allows us to apply Lemma \ref{l: f and d} and deduce $\left|\frac{d(p,p_1)}{s}-1\right|+\left|\frac{d(p,p_2)}{s}-\cos\frac{\alpha}{2}\right|<\epsilon(s)$.
Moreover, since $|\widetilde{\measuredangle}p_1pp_2-\frac{\alpha}{2}|<\epsilon(s)$ by Lemma \ref{l: angles between Gamma and gamma}, it follows that $\left|\widetilde{\measuredangle}pp_1p_2-(\frac{\pi}{2}-\frac{\alpha}{2})\right|\le\epsilon(s)$.
Choose $p',p'_2$ in the minimizing geodesics between $p,p_1$ and $p_1,p_2$ such that $d(p_1,p'_2)=d(p_1,p')=h_1(s)$. Then by angle comparison 
$\widetilde{\measuredangle}p'p_1p'_2\ge\widetilde{\measuredangle}pp_1p_2\ge\frac{\pi}{2}-\frac{\alpha}{2}-\epsilon(s)$, and hence $\partial B_N(p_1,h_1(s))\ge d(p',p'_2)\ge C^{-1}h_1(s)$. So by volume comparison we get
\begin{equation}\label{e: Nvolume}
    vol(B_N(p_1,h_1(s)))\ge C^{-1}\,h_1^2(s).
\end{equation}

On the other hand, let $\widetilde{M_0}\longrightarrow M_0:=M\setminus\Gamma$ be the universal covering, and $(\widetilde{M_0},\widetilde{g}(t),\widetilde{p}_1)$ be the pull-back Ricci flow of $(M_0,g(t),p_1)$, $t\in(-\infty,0]$, where $g(t)$ is the Ricci flow associated to $(M,g,p)$ with $g(0)=g$.
Then
$\widetilde{g}(0)=g_N+\varphi^2 d\theta^2$, $\theta\in(-\infty,\infty)$, and by using \eqref{e: Nvolume} we get
\begin{equation}
    vol(B_{\widetilde{g}(0)}(\widetilde{p}_1,h_1(s)))\ge \frac{1}{2}h_1(s)\,vol(B_N(p_1,\frac{1}{2}h_1(s)))\ge C^{-1}\,h^3_1(s).
\end{equation}
So by applying Corollary 45.1(b) in \cite{KL}, we obtain
$R(p_1)=R(\widetilde{p}_1)\le C\,h_1^{-2}(s)$.
\end{proof}

\begin{lem}\label{l: h_1/h_2 goes to 0}
Suppose $(M,g,p)$ is not a Bryant soliton.
Then $\frac{h_2(s)}{h_1(s)}\ri0$ as $s\rii$.
\end{lem}

\begin{proof}
Suppose by contradiction that there is a sequence $s_i\rii$ such that $\frac{h_2(s_i)}{h_1(s_i)}\ge C^{-1}>0$ for some $C>0$ and all $i$.
Let $\sigma_i$ be a minimizing geodesic from $\gamma(s_i)$ to some $q_i\in\Gamma$ such that $h_1(s_i)=d(\gamma(s_i),q_i)$. Then $\sigma_i$ intersects with $\Gamma$ orthogonally at $q_i$. 
Let $\Sigma_i=\phi^{-1}(\sigma_i)$, where $\phi:(M\setminus\Gamma,g)\rightarrow (N,g_N)$ is the Riemannian submersion. Then $(\Sigma_i,g_i)$ is a smooth rotationally symmetric surface with non-negative curvature, where $g_i$ is the metric induced by $g$.
Then by Theorem \ref{l: geometry at infinity}, $(\Sigma_i,R(\Gamma(s_i))g_i)$ smoothly converges to the time-0-slice of a non-compact ancient Ricci flow $g_{\infty}(t)$.

Moreover, by Theorem \ref{l: geometry at infinity} we know that any blow-down limit along $\Gamma$ is a product of $\mathbb{R}$ and a non-compact ancient Ricci flow, from which it follows that $\lim_{s\rii}h_1(s)R^{1/2}(\gs)=\infty$. This combining with
$\frac{h_2(s_i)}{h_1(s_i)}\ge C^{-1}$ and a volume comparison implies that the asymptotic volume ratio of $g_{\infty}(0)$ is positive, and hence $g_{\infty}(t)$ is flat, a contradiction.
\end{proof}

\begin{lem}\label{l: h_1h_2=s}
Suppose the asymptotic cone of $(M,g,p)$ is a ray.
Then there is some $C>0$ such that
$h_1(s)h_2(s)\ge C^{-1}s$ for all large $s$.
\end{lem}

\begin{proof}
The assertion clearly holds when $(M,g,p)$ is a Bryant soliton, so we may assume below that $(M,g,p)$ is not a Bryant soliton.

On the one hand, since $h_1(s)=d(\gamma(s),\Gamma)$, we have $d(q,\gamma(s))=h_1(s)$ for some $q\in\Gamma$. Let $\overline{q}$ be the image of $q$ under the $\mathbb{Z}_2$-action, and
$\sigma:[-\frac{1}{2}d(q,\overline{q}),\frac{1}{2}d(q,\overline{q})]$ be a minimizing geodesic from $q$ to $\overline{q}$. Then by the $\mathbb{Z}_2$-symmetry it follows that $\sigma$ intersects orthogonally with $\Sigma$ at $\sigma(0)$ and 
\begin{equation}
    d(q,\sigma(0))=d(q,\Sigma)=\frac{1}{2}d(q,\overline{q}).
\end{equation}
Moreover, by replacing $\sigma$ with its image under some $O(2)$-action, we may assume $\sigma(0)\in\gamma$.
So we have
\begin{equation}\label{e: le than h_1}
    \frac{1}{2}d(q,\overline{q})=d(q,\gamma)\le d(q,\gamma(s))= h_1(s).
\end{equation}

Since the asymptotic cone is a ray, by Lemma \ref{l: angles between Gamma and gamma} and $h_1(s)=d(\gamma(s),\Gamma)\le d(\gamma(s),\Gamma(s))$, we see $h_1(s)\le\epsilon(s)s$.
So by Lemma \ref{l: f and d} and using triangle inequality we obtain
\begin{equation}
    \begin{split}
        d(p,\sigma(0))&\le d(p,\gamma(s))+ d(\gamma(s),q)+d(q,\sigma(0))
        \le d(p,\gamma(s))+2 h_1(s)
        \le (1+\epsilon(s))s.
    \end{split}
\end{equation}
Suppose $\sigma(0)=\gamma(s')$ for some $s'>0$, then by Lemma \ref{l: f and d} this implies $s'\le(1+\epsilon(s))s$, which by the concavity of $h_2$ yields 
\begin{equation}\label{e: h_2}
    h_2(s)\ge(1-\epsilon(s))h_2(s')\ge\frac{1}{2}h_2(s').
\end{equation}

On the other hand, let $\Omega(s)\subset M$ be the domain bounded by $\phi^{-1}(\sigma)$, where $\phi:(M\setminus\Gamma,g)\rightarrow (N,g_N)$ is the Riemannian submersion, then 
\begin{equation}
    d(\partial\Omega(s),p)\ge d(p,\sigma(0))-d(q,\sigma(0))\ge (1-\epsilon(s))s-h_1(s)\ge (1-\epsilon(s))s,
\end{equation}
which implies $\Omega(s)\supset B(p,\frac{1}{2}s)$, 
So by Stokes' theorem, $R=\Delta f$, and Lemma \ref{l: int of R} we obtain
\begin{equation}\label{e: area}
    \textnormal{Area}(\partial\Omega(s))\ge\int_{\partial\Omega(s)}\langle\nabla f,\Vec{n}\rangle
    =\int_{\Omega(s)}\Delta f\,dvol_M\ge\int_{B(p,\frac{1}{2}s)}R\,dvol_M\ge C^{-1}\,s.
\end{equation}
By the $\mathbb{Z}_2$-symmetry we have $\frac{d}{dr}\mid_{r=0}\varphi(\sigma(r))=0$, which combining with the concavity of the warping function $\varphi$ implies
 $\varphi(\sigma(r))\le \varphi(\sigma(0))= h_2(s')$ for all $r\in[-\frac{1}{2}d(p,\overline{p}),\frac{1}{2}d(p,\overline{p})]$.
So 
\begin{equation}
    \textnormal{Area}(\partial\Omega(s))=\int_0^{2\pi}\int_{-\frac{1}{2}d(q,\overline{q})}^{\frac{1}{2}d(q,\overline{q})}\varphi(\sigma(r))\,dr\,d\theta\le 2\pi\,d(q,\overline{q})h_2(s')\le C h_1(s)h_2(s),
\end{equation}
where we used \eqref{e: le than h_1} and \eqref{e: h_2} in the last inequality. This together with \eqref{e: area} proves the lemma.

\end{proof}

\begin{lem}\label{l: alpha=0}
Suppose the asymptotic cone is a ray, and $\lim_{s\rii}h_2(s)<\infty$. Then $\lim_{s\rii}R(\gs)>0$.
\end{lem}

\begin{proof}
Suppose $s$ is sufficiently large, and assume $\lim_{r\rii}\varphi(\gamma(r))=\lim_{r\rii}h_2(r)=C$ for some $C>0$. Let $p_1=\Gamma(s)$, $p_2=\Gamma(-s)$, and $\sigma:[0,d(p_1,p_2)]\ri M$ be a minimizing geodesic from $p_1$ to $p_2$.
Let $pp_1,pp_2, p_1p_2=\sigma$ be minimizing geodesics between these points.
Then since $\widetilde{\measuredangle} p_1pp_2\le\epsilon(s)$, we have $\measuredangle pp_1p_2\ge\widetilde{\measuredangle}pp_1p_2\ge\frac{\pi}{2}-\epsilon(s)$.

For some $s'>>s$, take $q=\gamma(s')$, and let $qp_1,qp_2$ be minimizing geodesics between these point. By replacing $\sigma=p_1p_2$ and $pp_1$ with their image under suitable $O(2)$-actions, we may assume that $\measuredangle pp_1p_2+\measuredangle qp_1p_2\le\pi$.
Since by angle comparison $\measuredangle p_2p_1q\ge\widetilde{\measuredangle} p_2p_1q\ge\frac{\pi}{2}-\epsilon(s)$, it follows that $\left|\measuredangle pp_1p_2-\frac{\pi}{2}\right|\le\epsilon(s)$.
Note by Lemma \ref{l: f and d} we have $\measuredangle(\nabla f(p_1),pp_1)\le\epsilon(s)$, so by triangle inequality we obtain
\begin{equation}
   \left|\langle\nabla f,\sigma'(r)\rangle(0)\right|+\left|\langle\nabla f,\sigma'(r)\rangle(d(p_2,p_1))\right|\le\epsilon(s).
\end{equation}

By the dimension reduction Theorem \ref{l: geometry at infinity} we have $R^{-1/2}(\gs)<\frac{1}{2}d(p_1,p_2)$ and
\begin{equation}\label{e: r(s)}
    \varphi(\sigma(R^{-1/2}(\gs)))\ge C^{-1}R^{-1/2}(\gs).
\end{equation}
By the $\mathbb{Z}_2$-symmetry it follows that $\sigma$ intersects with $\Sigma$ orthogonally at $\sigma\left(\frac{1}{2}d(p_1,p_2)\right)$, and $\left.\frac{d}{dr}\right|_{r=\frac{1}{2}d(p_1,p_2)}\varphi(\sigma(r))=0$. So by the concavity of $\varphi$ we get
\begin{equation}
    \varphi(\sigma(R^{-1/2}(\gs)))\le\varphi\left(\sigma\left(\frac{1}{2}d(p_1,p_2)\right)\right)\le \lim_{r\rii}\varphi(\gamma(r))=C,
\end{equation}
which together with \eqref{e: r(s)} implies the lemma.
\end{proof}

Now we prove Theorem \ref{l: positive angle} of the equation $\lim_{s\rii}R(\Gamma(s))=\sin^2\frac{\alpha}{2}$.


\begin{proof}[Proof of Theorem \ref{l: positive angle}]
Without loss of generality we may assume $\Rm>0$, and $(M,g,f,p)$ is not a Bryant soliton, since the theorem clearly holds for $\mathbb{R}\times\cigar$ and the Bryant soliton.
We may also assume $R(p)=1$.

For each fixed $s$ sufficiently large, let  $\sigma:[0,d(p_1,p_2)]\ri M$ be a minimizing geodesic from $p_1=\Gamma(s)$ to $p_2=\Gamma(-s)$. 
By the soliton equation $\nabla^2 f=\Ric$ and by integration by parts we obtain
\begin{equation}\label{e: LHS}
    \langle\nabla f,\sigma'(r)\rangle\mid_{0}^{d(p_2,p_1)}=\int_0^{d(p_2,p_1)}\Ric(\sigma'(r),\sigma'(r))\,dr.
\end{equation}
First, we claim
\begin{equation}\label{e: sigma}
    \left|\langle\nabla f,\sigma'(r)\rangle\mid_{0}^{d(p_2,p_1)}-2|\nabla f|(\Gamma(s))\sin\frac{\alpha}{2}\right|\le\epsilon(s).
\end{equation}
If $\alpha=0$, the claim holds by Lemma \ref{l: alpha=0}.
So we may assume $\alpha>0$.

Let $p_3=\Gamma(2s)$, and $pp_2,pp_1,p_1p_2,p_1p_3,p_2p_3$ be minimizing geodesics between these points, where $p_1p_2=\sigma$ in particular.
On the one hand, by replacing geodesics $pp_1,p_1p_3$ with their images under suitable $O(2)$-actions (note $p,p_1,p_3\in\Gamma$ are fixed under $O(2)$-actions), we may assume $\measuredangle pp_1p_2+\measuredangle p_2p_1p_3\le\pi$.
On the other hand, by Lemma \ref{l: f and d} and Lemma \ref{l: angles between Gamma and gamma} we obtain
\begin{equation*}
    \left|\frac{d(p,p_1)}{s}-1\right|
    +\left|\frac{d(p,p_3)}{s}-2\right|
    +\left|\frac{d(p_1,p_2)}{s}-\sqrt{2-2\cos\alpha}\right|
    +\left|\frac{d(p_2,p_3)}{s}-\sqrt{5-4\cos\alpha}\right|\le\epsilon(s).
\end{equation*}
Since $\alpha>0$, we have $\sqrt{2-2\cos\alpha}>0$. So by the cosine formula we obtain 
\begin{equation}
    \left|\widetilde{\measuredangle}pp_1p_2-\frac{\pi-\alpha}{2}\right|+\left|\widetilde{\measuredangle}p_2p_1p_3-\frac{\pi+\alpha}{2}\right|\le\epsilon(s).
\end{equation}
Then by the angle comparison it follows that $\measuredangle pp_1p_2\ge\frac{\pi-\alpha}{2}+\epsilon(s)$ and $\measuredangle p_2p_1p_3\ge\frac{\pi+\alpha}{2}+\epsilon(s)$, which combining with $\measuredangle pp_1p_2+\measuredangle p_2p_1p_3\le\pi$ implies 
\begin{equation}
    \left|\measuredangle pp_1p_2-\left(\frac{\pi-\alpha}{2}\right)\right|\le\epsilon(s).
\end{equation}
Note by Lemma \ref{l: f and d} the angle between $\nabla f$ and the tangent vector of $pp_1$ at $p_1$ is smaller than $\epsilon(s)$, this implies
claim \eqref{e: sigma}.

Next, by the Dimension Reduction Theorem \ref{l: geometry at infinity}, $(M,R(\gs)g,\gs)$ is $\epsilon(s)$-close to $\mathbb{R}\times\cigar$, so we can find $D(s)< \min\{\frac{1}{2}d(p_2,p_1),\frac{1}{2}\epsilon(s)^{-1}\}$ such that $\lim_{s\rii}D(s)R^{1/2}(\Gamma(s))=\infty$. So it follows that $d(\sigma(D(s)),\Gamma)\ge\frac{1}{2}D(s)\cos\frac{\alpha}{2}$. 
Then by the same argument as in Lemma \ref{l: key estimate} we get $R\le C(D(s))^{-2}$ in the two metric balls of radius  $\frac{1}{2}\cos\frac{\alpha}{2}D(s)$ which are centered at $\sigma(D(s))$ and $\sigma(d(p_2,p_1)-D(s))$. This implies by the second variation formula that
\begin{equation}\label{e: D_3}
    \int_{D(s)}^{d(p_2,p_1)-D(s)}\Ric(\sigma'(r),\sigma'(r))\,dr\le\frac{C}{D(s)}
    \le\epsilon(s)R^{1/2}(\gs).
\end{equation}

If $\lim_{s\rii}R(\Gamma(s))=0$, by the uniform curvature bound for all dimension reductions we have 
\begin{equation}\label{e: D(s)_2}
    R^{-1/2}(\Gamma(s))\int_I\,\Ric(\sigma'(r),\sigma'(r))\,dr\le C,
\end{equation}
where $C>0$ is a constant independent of $s$, and $I=[0,D(s)]\cup[d(p_1,p_2)-D(s),d(p_1,p_2)]$ 
This combining with \eqref{e: sigma}\eqref{e: D_3} and \eqref{e: LHS} implies $\alpha=0$. So the theorem holds in this case.

If $\lim_{s\rii}R(\Gamma(s))>0$, the dimension reduction is a cigar soliton with scalar curvature equal to $1$ at the tip, and it follows that
\begin{equation}\label{e: D(s)}
    \left|\left(R^{-1/2}(\Gamma(s))\int_{I}\,\Ric(\sigma'(r),\sigma'(r))\,dr\right)-2\cos\frac{\alpha}{2}\right|\le\epsilon(s),
\end{equation}
where we used the fact that for a cigar soliton with the sectional curvature $K$ equal to $\frac{1}{2}$ at the tip, the integral of $K$ along a geodesic emanating from $p$ is $\int_0^{\infty}K\,dr=\int_0^{\infty}\frac{1}{2}\textnormal{sech}^2(\frac{1}{2} r)\,dr=1$.
This combining with \eqref{e: D_3} implies 
\begin{equation}\label{e: 222}
    \left|\int_0^{d(p_2,p_1)}\Ric(\sigma'(r),\sigma'(r))\,dr-2R^{1/2}(\Gamma(s))\cos\frac{\alpha}{2}\right|\le\epsilon(s).
\end{equation}
Combining \eqref{e: sigma}\eqref{e: 222} in \eqref{e: LHS} and letting $s\rii$ we obtain 
\begin{equation}
    \lim_{s\rii}|\nabla f|(\Gamma(s))\sin\frac{\alpha}{2}=\lim_{s\rii}R^{1/2}(\Gamma(s))\cos\frac{\alpha}{2}.
\end{equation}
By the identity $R+|\nabla f|^2=R(p)=1$, this implies $\lim_{s\rii}R^{1/2}(\Gamma(s))=\sin\frac{\alpha}{2}$ and $\lim_{s\rii}|\nabla f|(\Gamma(s))=\cos\frac{\alpha}{2}$, which proves the theorem.

\end{proof}

Corollary \ref{c: aymptotic geometry} follows
immediately from Theorem \ref{l: positive angle} and Theorem \ref{l: geometry at infinity}.

Now we prove Theorem \ref{t: a ray implies Bryant soliton} by a bootstrap argument: 
First, since $g=g_N+\varphi^2d\theta^2$ on $M\setminus\Gamma$, the vector field $\frac{\partial}{\partial\theta}$ is a killing field.
Then by the killing equation we can establish the following relation between the Ricci curvature and the warping function $\varphi$, when they are restricted on $\gamma\subset\Sigma$:
\begin{equation}\label{e: good}
    \Ric\left(\frac{\partial}{\partial\theta},\frac{\partial}{\partial\theta}\right)=|\nabla f|(\gamma(s))h_2(s)h'_2(s).
\end{equation}
Recall we define $h_2(s)=\varphi(\gamma(s))$.

Suppose that the soliton is not a Bryant soliton, then by combining the estimates from Lemma \ref{l: key estimate}-\ref{l: h_1h_2=s} in the equation \eqref{e: good}, we
obtain that $h_2(s)<<s^{1/2}$. Replacing Lemma \ref{l: h_1/h_2 goes to 0} with this new upper bound, then the same argument shows that $h_2(s)\le C$. 
This implies $\lim_{s\rii}R(\gs)>0$, and by Theorem \ref{l: positive angle} we obtain a contradiction.

\begin{proof}[Proof of Theorem \ref{t: a ray implies Bryant soliton}]
Let $\epsilon(s)$ be constants that converge to $0$ as $s\rii$, and let $C$ denote all constants that are uniform for all large $s$.
Suppose by contradiction that $M$ is not a Bryant soliton. We shall use the notations in Lemma \ref{l: key estimate}-\ref{l: h_1h_2=s}.
Since $g=g_N+\varphi^2d\theta^2$ on $M\setminus\Gamma$, it follows that
$X:=\frac{\partial}{\partial\theta}$ is a killing field.
So by the identity of killing field we have
\begin{equation}
    \langle\nabla_X X,\nabla f\rangle +
    \langle\nabla_{\nabla f}X,X\rangle=0.
\end{equation}
Note that $\langle X,\nabla f\rangle=0$ and $\nabla^2 f=\Ric$, this gives the identity
\begin{equation}\label{e: key identity}
    \Ric\left(\frac{X}{|X|},\frac{X}{|X|}\right)=\frac{\nabla f(|X|)}{|X|}.
\end{equation}

Restrict the LHS of \eqref{e: key identity} on $\gamma(s)$ and abbreviate it by $\widetilde{R}(s)$. Then by the relations among $h_1(s),h_2(s)$ and $R(\gamma(s))$ from Lemma \ref{l: h_1h_2=s}, \ref{l: h_1/h_2 goes to 0}, and \ref{l: key estimate} we obtain
\begin{equation}
    s\,\widetilde{R}(s)\le s\,R(\gamma(s))\le Ch_1(s)h_2(s)R(\gamma(s))\le\epsilon(s)h_1(s)^2R(\gamma(s))\le\epsilon(s),
\end{equation}
which by \eqref{e: key identity}, $\lim_{s\rii}|\nabla f|(\gamma(s))=C>0$, and $h'_2(s)\ge 0$ implies
\begin{equation}
    \frac{h'_2(s)}{h_2(s)}\le\frac{\nabla f(h_2(s))}{|\nabla f|\cdot h_2(s)}=\frac{2\widetilde{R}(s)}{|\nabla f|}<\frac{\epsilon(s)}{Cs}<\frac{\epsilon_0}{s},
\end{equation}
for all large $s$ and some $\epsilon_0\in(0,\frac{1}{2})$. So $h_2(s)<C s^{\epsilon_0}$ for all large $s$.

Next, by using $h_2(s)<C s^{\epsilon_0}$ and applying Lemma \ref{l: h_1h_2=s} again we obtain $h_1(s)\ge C^{-1}s^{1-\epsilon_0}$,
which combining with Lemma \ref{l: key estimate} again gives
\begin{equation}
   \widetilde{R}(s)\le R(\gamma(s))\le Cs^{-2+2\epsilon_0}.
\end{equation}
Now substituting this into equation \eqref{e: key identity} we obtain
\begin{equation}
    \frac{h'_2(s)}{h_2(s)}< Cs^{-2+2\epsilon_0},
\end{equation}
which implies $h_2(s)< Ce^{-Cs^{-1+2\epsilon_0}}$,
and hence $\lim_{s\rii}h_2(s)<\infty$. 
This by Lemma \ref{l: alpha=0} implies $\lim_{s\rii}R(\gs)>0$, which by Theorem \ref{l: positive angle} yields a contradiction.
\end{proof}

Corollary \ref{c: unique of bryant} follows directly from Theorem \ref{t: a ray implies Bryant soliton}. It remains to prove Corollary \ref{c: alpha_i}.

\begin{proof}[Proof of Corollary \ref{c: alpha_i}]
First, by the proof of Theorem \ref{t: existence with prescribed eigenvalue} and Theorem \ref{t: a ray implies Bryant soliton} there exists a sequence of $\mathbb{Z}_2\times O(2)$-symmetric 3d expanding gradient solitons with positive curvature operator $\{(M_{1k},g_{1k},p_{1k})\}_{k=1}^{\infty}$, which smoothly converges to a 3d flying wing $(M_1,g_1,p_1)$. 
We may assume $R_{g_{1k}}(p_{1k})=R_{g_1}(p_1)=1$, and the asymptotic cone of $(M_1,g_1,p_1)$ is a sector with angle $\alpha_1\in(0,\pi)$.
This by Theorem \ref{l: positive angle} implies $\lim_{s\rii}R_{g_1}(\Gamma(s))=\sin^2\frac{\alpha_1}{2}$.

Let $(M_0,g_0,p_0)$ be a Bryant soliton with $R_{g_0}(p_0)=1$, since $\lim_{s\rii}R_{g_0}(\Gamma(s))=0$, we can find $s_1>0$ such that $R_{g_0}(\Gamma(s_1))<\frac{1}{2}\sin^2\frac{\alpha_1}{2}$.
Choose a constant $\widehat{R}\in(R_{g_0}(\Gamma(s_1)),\frac{1}{2}\sin^2\frac{\alpha_1}{2})$. Then by the convergence to $(M_1,g_1,p_1)$ and the continuity argument in Theorem \ref{t: existence with prescribed eigenvalue}, we can find a sequence of $\mathbb{Z}_2\times O(2)$-symmetric expanding gradient solitons $(M_{2k},g_{2k},p_{2k})$ with positive curvature operator, which smoothly converges to a 3d flying wing $(M_2,g_2,p_2)$, with $R_{g_2}(p_2)=R_{g_{2k}}(p_{2k})=1$ and  $R_{g_2}(\Gamma(s_1))=\widehat{R}$. Assume the asymptotic cone of $(M_2,g_2,p_2)$ is a sector with angle $\alpha_2\in[0,\pi]$. Then $\alpha\in(0,\pi)$ by Theorem \ref{t: a ray implies Bryant soliton}. Moreover, by Theorem \ref{l: positive angle} we have 
\begin{equation}
    \sin^2\frac{\alpha_2}{2}=\lim_{s\rii}R_{g_2}(\Gamma(s))\le\widehat{R}<\frac{1}{2}\sin^2\frac{\alpha_1}{2}.
\end{equation}

Therefore, by induction we obtain a sequence of 3d flying wings $(M_i,g_i,p_i)$ whose asymptotic cone is a sector with angle $\alpha_{i}$ satisfying $\sin^2\frac{\alpha_{i+1}}{2}<\frac{1}{2}\sin^2\frac{\alpha_{i}}{2}$ for all $i$. So $\alpha_i\ri0$ as $i\rii$.
\end{proof}

\end{section}

\bibliography{bib}
\bibliographystyle{abbrv}

\end{document}